\documentclass[article,10pt,reqno]{amsart}
\usepackage{amscd,amssymb,graphicx,color,epigraph}

\usepackage[shortalphabetic,initials]{amsrefs}

\newcommand{\red}[1]{{\color{red} #1}}

\usepackage[dvipsnames]{xcolor}
\usepackage[T1]{fontenc}
\usepackage{hyperref}
\hypersetup{
colorlinks = True,
linkcolor=Bittersweet,
citecolor=red,
}
\usepackage{enumitem}
\setlist[itemize]{leftmargin=18pt}
\setlist[enumerate]{leftmargin=18pt}
\usepackage{comment}

\usepackage{amssymb}
\usepackage{amsthm}
\usepackage{bm}
\usepackage{latexsym}
\usepackage{amsmath}
\usepackage{hyperref}
\usepackage{eufrak}
\usepackage{mathrsfs}
\usepackage{amscd}
\usepackage[all,cmtip]{xy}
\usepackage{leftindex}
\usepackage{color}
\usepackage{amscd}
\usepackage{listings}
\usepackage[center]{caption}
\theoremstyle{plain}
\usepackage{dirtytalk}

 \numberwithin{equation}{section}
 \newcommand{\eni}{k1A}
\newcommand{\enii}{k2A}

\newtheorem{theorem}{Theorem}[section]
\newtheorem{proposition}[theorem]{Proposition}

\newtheorem{question}[theorem]{Question}
\newtheorem{lemma}[theorem]{Lemma}

\newtheorem{corollary}[theorem]{Corollary}

\newtheorem{conjecture}[theorem]{Conjecture}

\theoremstyle{definition}

\newcommand{\appsection}[1]{\let\oldthesection\thesection
\renewcommand{\thesection}{Appendix \oldthesection}
\section{#1}\let\thesection\oldthesection}

\newtheorem{definition}[theorem]{Definition}

\newtheorem{remark}[theorem]{Remark}
\newtheorem{example}[theorem]{Example}

\DeclareMathOperator{\Cl}{Cl}
\DeclareMathOperator{\Pic}{Pic}

\DeclareMathOperator{\ext}{ext}
\DeclareMathOperator{\Ext}{Ext}

\DeclareMathOperator{\exc}{Exc}

\DeclareMathOperator{\Hom}{Hom}
\DeclareMathOperator{\rk}{rank}

\def\cA{{\mathcal{A}}}

\def\cE{{\mathcal{E}}}

\def\D{{\mathbb{D}}}
\def\R{{\mathbb{R}}}
\def\Z{{\mathbb{Z}}}
\def\F{{\mathbb{F}}}
\def\Q{{\mathbb{Q}}}
\def\C{{\mathbb{C}}}
\def\P{{\mathbb{P}}}

\def\setW{{\mathbb{W}}}

\def\O{{\mathcal{O}}}
\def\Wa{{\texttt{W}}}
\def\W{{\mathcal{W}}}
\def\bW{{\overline{W}}}

\makeatletter
\providecommand{\leftsquigarrow}{%
  \mathrel{\mathpalette\reflect@squig\relax}%
}
\newcommand{\reflect@squig}[2]{%
  \reflectbox{$\m@th#1\rightsquigarrow$}%
}
\makeatother

\title{Wahl singularities in degenerations of del Pezzo surfaces}

\author{Giancarlo Urz\'ua}
\address{Facultad de Matem\'aticas,
Pontificia \allowbreak Universidad \allowbreak{Cat\'olica} de Chile, Santiago, Chile.}

\email{gianurzua@gmail.com}

\author{Juan Pablo Z\'u\~niga}
\address{Facultad de Matem\'aticas,
Pontificia Universidad Cat\'olica de Chile, Santiago, Chile.}
\email{jpzuniga3@uc.cl}

\date{\today}

\begin{document}

\begin{abstract}
For any fixed $1 \leq \ell \leq 9$, we characterize all Wahl singularities that appear in degenerations of del Pezzo surfaces of degree $\ell$. This extends the work of Manetti and Hacking–Prokhorov in degree $9$, where Wahl singularities are classified using the Markov equation. To achieve this, we introduce del Pezzo Wahl chains with markings. They define marked del Pezzo surfaces $W_{*m}$ that govern all such degenerations. We also prove that every marked del Pezzo surface degenerates into a canonically defined toric del Pezzo surface with only T-singularities. In addition, we establish a one-to-one correspondence between the $W_{*m}$ surfaces and certain fake weighted projective planes. As applications, we show that every Wahl singularity occurs for del Pezzo surfaces of degree $\leq 4$, but that there are infinitely many Wahl singularities that do not arise for degrees $\geq 5$. We also use Hacking's exceptional collections to provide geometric proofs of recent results by Polishchuk and Rains on exceptional vector bundles on del Pezzo surfaces.
\end{abstract}

%\dedicatory{To the memory of Martin Aigner}

%\keywords{Markov number, birational geometry, Hirzebruch-Jung continued fraction, Wahl singularity}

\maketitle

%----------------------------------------------------------------------------

\section{Introduction} \label{s0}

Let $W$ be a projective surface with only Wahl singularities, i.e. cyclic quotient singularities of type $\frac{1}{n^2}(1, na - 1)$ where $\gcd(n, a) = 1$. Assume that $K_W^2 > 0$. Then $K_W$ or $-K_W$ is big. Consider one-parameter $\mathbb{Q}$-Gorenstein deformations $W_t$ of $W$, denoted by $W_t \rightsquigarrow W_0 = W$, where $t \in \D = \{ t \in \C \colon |t| \ll \epsilon \}$, and the 3-fold total space $\mathcal{W}$ of the deformation has $K_{\W}$ $\mathbb{Q}$-Cartier. If $W_t$ is nonsingular for $t \neq 0$, then the deformation induces $\mathbb{Q}$HD (rational homology disk) smoothings at each of the singularities of $W$. In this case, $W_t$ degenerates into $W$ in the mildest possible way. For example, we have that $K_{W_t}^2$, $\chi_{\text{top}}(W_t)$, $p_g(W_t)$, and $q(W_t)$ are constant, and $H_2(W_t, \mathbb{Z}) \subseteq H_2(W, \mathbb{Z})$ with finite index.

%\bigskip 

When $K_W$ is big, the nonsingular surfaces $W_t$ are of general type. Classifying all possible surfaces $W$ for a given $K_W^2$ is a difficult task. For instance, there are well-known open questions in relation to geography of surfaces of general type, and exotic 4-manifolds homeomorphic to blow-ups of $\P^2$ at few points \cite{RU21,RU22}. 

%\bigskip 

When $-K_W$ is big, the nonsingular surfaces $W_t$ are rational, and thus deformations of del Pezzo surfaces. One particular characteristic in this case, which is not necessarily true when $K_W$ is big, is that there are no local-to-global obstructions to deform $W$. Hence, any such $W$ must have $\Q$-Gorenstein smoothings and partial smoothings. For each singularity of $W$, we consider surfaces $W_*$ with only one singularity which are degenerations of del Pezzo surfaces.

Building on the work of B\v{a}descu \cite{B86}, Manetti \cite{Ma91} and Hacking \cite{H04}, Hacking and Prokhorov \cite{HP10} classified all possible $W_*$ when $W_t=\P^2$. They proved that every $W_*$ is a $\Q$-Gorenstein partial smoothing of $\P(x^2,y^2,z^2)$, where $(x,y,z)$ satisfies the Markov equation $x^2+y^2+z^2=3xyz$. This characterization motivates the following more general question.  

\begin{question}
Is there a classification of surfaces $W_*$ in a $\mathbb{Q}$-Gorenstein smoothing $W_t \rightsquigarrow W_*$, where the general fiber $W_t$ is a del Pezzo surface of degree $K_{W_t}^2$, and $W_*$ has a single Wahl singularity? Are there constraints on the types of Wahl singularities that can appear for a fixed degree?
\label{question1}
\end{question}

In this paper, we provide a complete answer to Question \ref{question1}. Further research on $\Q$-Gorenstein degenerations of del Pezzo surfaces can be found in \cite{P15}, \cite{Urz16a}, \cite{akhtar16}, \cite{P19}, \cite{DVS24}, \cite{UZ23}, \cite{Pe25}. See \cite{Urz25} for a background on degenerations of surfaces with only Wahl singularities. 

\vspace{0.2cm}
In the following, we explain Hacking-Prokhorov's result in a way that connects to the results of this paper. We start with some common terminology. Let $0<q<m$ be coprime integers. The corresponding cyclic quotient singularity (c.q.s.) $\frac{1}{m}(1,q)$ has a minimal resolution that replaces the singular point by a chain of curves $E_i$ such that $E_i \simeq \P^1$ and $E_i^2=-e_i \leq -2$. The integers $e_i$ are encoded in the Hirzebruch-Jung continued fraction $$\frac{m}{q}=e_1 - \frac{1}{e_2 - \frac{1}{\ddots - \frac{1}{e_r}}}=:[e_1,\ldots,e_r].$$ By definition, its dual chain is $[b_1,\ldots,b_s]=\frac{m}{m-q}$ where $b_i \geq 2$. We have that $$[e_1,\ldots,e_r,1,b_s,\ldots,b_1]=0.$$ Wahl singularities are precisely the quotient singularities that admit a smoothing with second Betti number equal to zero, i.e. a $\Q$HD smoothing. A Wahl chain is the Hirzebruch-Jung continued fraction of a Wahl singularity. If $\Wa$ is a Wahl chain for some $0<a<n$, then its dual chain $\Wa^{\vee}$ corresponds to $\frac{n^2}{n^2-na+1}=[x_1,\ldots,x_u,2,y_v,\ldots,y_1]$ where $\frac{n}{a}=[y_1,\ldots,y_v]$ and $\frac{n}{n-a}=[x_1,\ldots,x_u]$. In fact, we have $\Wa=[y_1,\ldots,y_{v-1},y_v+x_u,x_{u-1},\ldots,x_1]$.

\vspace{0.2cm} 

Let us explain Hacking-Prokhorov's theorem in terms of Wahl chains. Let $\frac{n^2}{na-1}=[e_1,\ldots,e_r]$ be a Wahl chain. Then (see for example \cite[Proposition 4.1]{UZ23}) this is the Wahl chain of a $W_*$ of degree $9$ if and only if $$[e_1,\ldots,e_r]= [\Wa_0^{\vee},\alpha,\Wa_1^{\vee}]$$ for some Wahl chains $\Wa_i$ and some $\alpha \geq 2$. In fact, if $\frac{n_0^2}{n_0a_0-1}=\Wa_0$ and $\frac{n_1^2}{n_1 a_1-1}=\Wa_1$, then $n_0^2+n_1^2+n^2=3 n_0 n_1 n$. We have $3$ possibilities for $\alpha$:
\begin{itemize}
    \item if both $\Wa_i$ are not empty, then $\alpha=10$;
    \item if only one $\Wa_i$ is empty, then $\alpha=7$;
    \item if both $\Wa_i$ are empty, then $\alpha=4=n^2$.
\end{itemize}

Let us simplify this characterization with the following definition.

\begin{definition}
We say that $[f_1,\ldots,f_r]$ with $f_i \geq 2$ admits a \textit{zero continued fraction of weight $\lambda$}
if there are indices $i_1<i_2<\ldots<i_v$ for some $v \geq 1$ and integers $d_{i_k}\geq 1$ such that $$[\ldots,f_{i_1}-d_{i_1},\ldots,f_{i_2}-d_{i_2},\ldots,f_{i_v}-d_{i_v},\ldots]=0,$$  and $\lambda+1=\sum_{k=1}^v d_{i_k}$.
\label{zcfSigma}
\end{definition}

It is easy to show that $[f_1,\ldots,f_r]$ admits a zero continued fraction of weight $0$ if and only if $[f_1,\ldots,f_r]$ is a dual Wahl chain. In this way, Hacking-Prokhorov's result becomes the next equivalence. 

\begin{theorem}
A Wahl chain $\Wa=[e_1,\ldots,e_r]$ appears in a degeneration of $\P^2$ if and only if there is $i \in \{1,\ldots,r \}$ such that $[e_1,\ldots,e_{i-1}]$ and $[e_{i+1},\ldots,e_{r}]$ admit zero continued fractions of weight $0$. (One ($i=1,r$) or both ($r=1$) chains could be empty.)
\label{markov}
\end{theorem}

Our answer to Question \ref{question1} is Theorem \ref{mainthm1}, which is formulated in the spirit of Theorem \ref{markov}. The Wahl singularities that appear in degenerations of del Pezzo surfaces also depend on "markings" over the corresponding Wahl chain.

%We describe all Wahl singularities that appear in degenerations of del Pezzo surfaces of a given degree. This depends not only on the Wahl singularity, but also on "markings" over the Wahl chain. %, and minimality of $W_*$ has to do with "minimal markings". 

\begin{definition}
A Wahl chain $[e_1,\ldots,e_r]$ is \textit{del Pezzo of type (I) or (II)} if it satisfies one of the following:
\begin{itemize}
%\item[(1)] It admits a zero continued fraction of weight $\lambda \leq 8$. 
\item[(I)] $[e_1,\ldots,e_{r-1}]$ or $[e_{2},\ldots,e_r]$ admits a zero continued fraction of weight $\lambda \leq 8$. 
\item[(II)] There is $i \in \{2,\ldots,r-1\}$ such that $[e_1,\ldots,e_{i-1}]$ and $[e_{i+1},\ldots,e_r]$ admit zero continued fractions of weights $\lambda_1$ and $\lambda_2$ such that $\lambda_1 + \lambda_2 \leq 8$.
\end{itemize}
Its \textit{degree} is either $9-\lambda$ for (I), or $9-\lambda_1-\lambda_2$ for (II). Its \textit{marking} is the data of the zero continued fraction(s) involved in its type. If $[k_1,\ldots,k_{i-1}]$ and $[k_{i+1},\ldots,k_r]$ are the corresponding zero continued fractions, then its marking is $$[k_1,\ldots,k_{i-1},\underline{e_i},k_{i+1},\ldots,k_r].$$ We call $\underline{e_i}$ the central mark.
\label{DPWahlchain}
\end{definition}

All markings for a given Wahl chain can be found here \href{https://colab.research.google.com/drive/1tooXaHZnug0nCRvmXs9g_WIi1jFm4dgs?usp=sharing}{MarkingsWahlChains}. 

\vspace{0.3cm} 

A del Pezzo Wahl chain with a given marking defines del Pezzo surfaces $W_{*m}$ with one singularity (Definition \ref{def: markedsurface}). We call them \textit{marked del Pezzo} surfaces. Its minimal resolution $\phi \colon X \to W_{*m}$ has a fibration $X \to \P^1$ of genus $0$ such that $\exc(\phi)$ is formed by one section of $X \to \P^1$, and contains components of one fiber for type (I), or two fibers for type (II). Specifically, in each case it contains one or two fibers minus all $(-1)$-curves in the fiber(s). In Figure \ref{f0} we show diagrams of this, where $\exc(\phi)$ is black and $(-1)$-curves are red.

\begin{figure}[htbp]
\centering
\includegraphics[height=5cm, width=9.3cm]{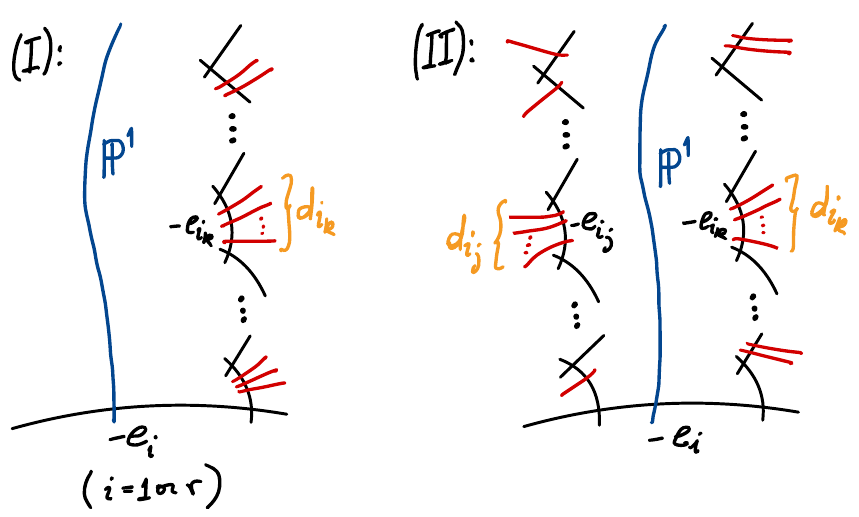}
\caption{Minimal resolutions of $W_{*m}$ for types (I) and (II).} 
\label{f0}
\end{figure}

Given a $\Q$-Gorenstein smoothing $W_t \rightsquigarrow W$ with $K_W$ not nef, the minimal model program (MMP) relative to $\D$ gives us two types of birational operations: flips and divisorial contractions. The case of divisorial contractions is particularly recurrent when we work under the assumption $-K_W$ big. There are two types of basic divisorial contractions: W-blow-downs and Iitaka-Kodaira blow-downs; see Definition \ref{divcontr}. The W-blow-downs come in infinite families over a Wahl singularity, and its numerical data belong to a Mori train (Definition \ref{moritrain}). An Iitaka-Kodaira blow-down is defined by a $(-1)$-curve in $W$ which is disjoint from the singularities of $W$. It induces a contractible divisor of $(-1)$-curves in $W_t \rightsquigarrow W$, so the contraction is a new $\Q$-Gorenstein smoothing $W'_t \rightsquigarrow W'$. We will use the classical notation $\text{Bl}_u(W)$ for the blow-up of $u$ nonsingular points in $W$. Following \cite{CH17}, we call the corresponding $(-1)$-curves \textit{floating curves}.       

\begin{theorem}
Let $W$ be a surface with only Wahl singularities such that $\ell=K_W^2>0$ and $-K_W$ is big. For each singularity $P \in W$, there exists a marked surface $W_{*m}$ of degree $\ell^\prime\geq\ell$ containing a singularity of the same type as $P$, such that $\text{Bl}_{\ell'-\ell}(W_{*m})$ admits a $\Q$-Gorenstein degeneration to $W$.

In particular, for any fixed degree $\ell$, every Wahl singularity appearing in a $\Q$-Gorenstein degeneration of a del Pezzo surface is realized by a del Pezzo Wahl chain of degree $\ell$, or by one of degree $\ell'\geq\ell$ obtained by adjoining floating curves. \footnote{For $\frac{1}{4}(1,1)$ there is one extra option, constructed from the blow-up of five points over a line in $\P^2$.}  
\label{mainthm1}
\end{theorem}

%\begin{remark} 
%A marked del Pezzo surface $W_{*m}$ may contain floating curves. Consider the contraction of all of them, including any new ones that may arise. The resulting surface has the same Wahl singularity and is a degeneration of a marked del Pezzo surface $W'_{*m}$, which leads to a divisorial contraction $W_{*m} \to W'_{*m}$. Furthermore, through successive W-blow-downs, we may obtain a contraction $W'_{*m} \to W''_{*m}$. The surface $W''_{*m}$ has a new Wahl singularity over which $W'_{*m}$ is constructed through a composition of various Mori trains. The minimally marked surfaces $W''_{*m}$ with no floating curves really represent the new marked del Pezzo surfaces for each degree. %{\color{blue} Creo que este procedimiento coincide con estas cascadas de DongSeon, quizás sería bueno mencionarlo aquí}
%\end{remark}

%\blue{Thus, along the lines of Question \ref{question1} every $W_*$ with $K_{W_*}^2>0$ and $-K_{W_*}$ big. is deformation equivalent to a $Bl_u(W_{*m})$ for some $u\geq 0$. The preceding theorem builds on the classification of the exceptional locus of the minimal resolution of $W_*$; see Theorem \ref{Thm class}.}

In degree $8$ we have degenerations of two del Pezzo surfaces: $\mathbb{F}_0$ and $\mathbb{F}_1$. The following theorem distinguishes them in terms of the corresponding $W_{*m}$; see \ref{ss5.1}.

\begin{theorem}
Let $W_{*m}$ be a marked surface of degree $8$ with singularity $\frac{1}{n^2}(1,na-1)$. 

\begin{itemize}
    \item[(0)] There is a degeneration $\F_0 \rightsquigarrow W_{*m}$ if and only if $n$ is odd, and $n \Gamma \cdot K_{W_{*m}}$ is even for every marking curve $\Gamma$. %In this case, $W_{*m}$ admits neither floating curves nor W-blow-downs.   
    \item[(1)] There is a degeneration $\F_1 \rightsquigarrow W_{*m}$ if and only if one of the following holds: $W_{*m}$ is a $W$-blow-up or the blow-up at a smooth point of a marked surface $W'_{*m}$ of degree $9$; $n$ is even; or $n$ is odd, and there exists a marking curve $\Gamma$ such that $n\Gamma \cdot K_{W_{*m}}$ is odd.
\end{itemize}
\label{mainthm2}
\end{theorem}

%\blue{quizás sacamos esta frase: As mentioned before, the surfaces $W_{*m}$ of degree $9$ are precisely partial smoothings of a toric del Pezzo surface $\P(x^2,y^2,z^2)$ where $x^2+y^2+z^2=3xyz$}. 

Analogously to Hacking-Prokhorov's theorem, we also prove that every marked del Pezzo surface $W_{*m}$ is a partial smoothing of a naturally defined toric del Pezzo surface $W_{*m}^T$ with only T-singularities. We recall that T-singularities are Du Val singularities or $\frac{1}{dn^2}(1,dna-1)$ where gcd$(n,a)=1$ and $d\geq 1$. To prove this, we introduce in \ref{s3.4} the notion of \textit{slidings} of the marking curves in $W_{*m}$.

\begin{theorem}
A marked del Pezzo surface $W_{*m}$ defines a toric del Pezzo surface $W_{*m}^T$ with only T-singularities so that $W_{*m}$ is a $\Q$-Gorenstein partial smoothing of $W_{*m}^T$. 
\label{mainthm3}
\end{theorem}

In this way, every $W_{*m}$ is deformation equivalent to a toric del Pezzo surface. By \cite[Theorem 6]{KNP17} all toric del Pezzo surfaces with T-singularities are obtained through particular families over $\P^1$. This deformations are induced by a combinatorial process known as mutation of Fano polygons (see \cite[Definition 2]{KNP17} and \cite[Lemma 7]{akhtar16}). In \ref{s3.4} we show that slidings define two types of deformations over $\mathbb{D}$, which we refer to as \textit{mutations}, on a given $W_{*m}$. The difference between them lies in the self-intersection of the marking curve $\Gamma$ of the slide. It holds that $\Gamma \cdot K_{W_{*m}}<0$, but the number $\Gamma^2$ is non-zero. When $\Gamma^2>0$ we refer to them as \textit{Markov mutations}. In degree $9$, all slidings are of this form. On the other hand, when $\Gamma^2<0$ we refer to them as \textit{Mori mutations}, and they are represented by two consecutive wagons in a Mori train.  

%Indeed, the surfaces $W_{*m}$ lie in the class TG, in the sense \cite[Definition 1]{akhtar16}, and the polygon associated to the toric surface $W_{*m}^T$ of Theorem \ref{mainthm3} belongs to one of such mutation equivalence class. 

For any $W_{*m}$, there exists a toric surface $\widehat{W}_{*m}$ with a crepant birational morphism $\widehat{W}_{*m}\to W_{*m}^T$ (see Section \ref{s4}). The surface $\widehat{W}_{*m}$ admits a contraction to a certain fake weighted projective plane (see Section \ref{s7}). The following theorem characterizes them (see also Remark \ref{type(I)Diophantine}).  

\begin{theorem}
Let $P(n^2,m_1,m_2)$ be a toric surface of Picard number $1$, and singularities $\frac{1}{n^2}(1,na-1)$, $\frac{1}{m_1}(1,m_1-q_1)$, $\frac{1}{m_2}(1,m_2-q_2)$ forming the chain $$\frac{1}{m_1}(1,m_1-q_1) - (1) - \Big[ {n \choose a} \Big] - (1) - \frac{1}{m_2}(1,m_2-q_2^{-1}),$$ where $0<a<n$ and $0<q_i<m_i$ are coprime. Assume that
\begin{itemize}
    \item[(1)] $q_1m_2+q_2m_1+n^2=m_1 m_2 d$, for some $d\geq 2$. 
    \item[(2)] $m_1+m_2=n(m_1a-nq_1^{-1})=n(m_2(n-a)-nq_2^{-1})$. 
    \item[(3)] $\frac{1}{m_1}(1,q_1)$ and $\frac{1}{m_2}(1,q_2)$ admit zero continued fractions of weights $\lambda_1$ and $\lambda_2$, respectively, with $\ell=9-\lambda_1-\lambda_2>0$.
\end{itemize}
Then there is a partial resolution $W \to P(n^2,m_1,m_2)$ such that $W=\widehat{W}_{*m}$ for some marked surface $W_{*m}$ with Wahl singularity $\frac{1}{n^2}(1,na-1)$ and degree $\ell$.
\label{mainthm6}
\end{theorem}

To recover $W_{*m}$ from $P(n^2,m_1,m_2)$, we run the MMP over suitable P-resolutions \cite[Definition 3.8]{KSB88}. Hence, to find the possible singularity indices in a given degree, it is essential to understand certain cyclic quotient singularities admitting P-resolutions. In degree $8$, this boils down to understanding those who admit extremal P-resolutions \cite[\S 4]{HTU17}. Theorem \ref{mainthm6} shows that special Diophantine equations must hold in degree $8$ (see Example \ref{degree8} and Remark \ref{beatdegree8}). Degree $9$, which is too special, gives the Markov equation (see Remark \ref{markovequation}). In general, particular inequalities must hold: $(n+a)^2 \geq \ell (na-1)$ for type (I) and $(n^2+m_1+m_2)^2 \geq \ell n^2 m_1 m_2$ for type (II); see Remark \ref{hodgeIndex}.

Theorem \ref{mainthm1} leads to several consequences. First, in Lemma \ref{canmarking} we find that every Wahl chain admits a canonical marking so that it is del Pezzo of type (I) and degree $4$. Moreover, Theorem \ref{theorem:W*m} shows that this marked surface $W_{*m}$ arises as a degeneration of a smooth del Pezzo surface of degree $4$. Indeed, in the spirit of Question \ref{question1}, we obtain the following two theorems:

\begin{theorem}
Every Wahl singularity is realizable in a $\Q$-Gorenstein smoothing $W_t\rightsquigarrow W_*$, where $W_t$ is a del Pezzo surface of degree $\ell\leq 4$.   
\label{mainthm4}
\end{theorem}

%Based on Theorem \ref{mainthm1}, we show that this is not the case in higher degree.

\begin{theorem}
There exist infinitely many Wahl singularities that are not realizable in a $\Q$-Gorenstein smoothing $W_t\rightsquigarrow W$, where $W_t$ is a del Pezzo surface of degree $\ell\geq 5$. 
\label{mainthm5}
\end{theorem}

For many concrete examples that are not allowed in degrees $\geq 5$, see Proposition \ref{5isnotfree} and Remark \ref{muchmore}. If we blow-up points out of the singularity, then we obtain del Pezzo surfaces of lower degrees degenerating to a surface with the same singularity. Alternatively, one can use \textit{W-blow-ups} for a similar purpose (see Definition \ref{divcontr}). These modify a given Wahl singularity into any Wahl singularity appearing in a Mori train over it. However, not every Wahl singularity arises in this way. On the other hand, using specific markings, we observe that the famous Markov's Conjecture \ref{ConjMarkov} is false in degrees $\geq 8$ as we remark in Example \ref{nomarkov}.

We conclude the paper with applications to exceptional collections of vector bundles on del Pezzo surfaces. The key point is that marked surfaces naturally produce full exceptional collections via Hacking’s vector bundles \cite{H13}.

\begin{theorem}[Corollary \ref{HecW*m}]
Let $W_{*m}$ be a marked surface that is not of type (I) and degree $1$ (see Lemma \ref{wirdo}). Consider a general $\Q$-Gorenstein smoothing $W_t \rightsquigarrow W_{*m}$ so that $W_t$ is a del Pezzo surface. Then, $W_t$ admits an explicit Hacking exceptional collection induced by $W_t \rightsquigarrow \widehat{W}_{*m}$ that is full and strong.   
\label{mainthm7}
\end{theorem}

The construction is based on Theorem \ref{H.e.c.theorem}, as developed in \cite{TU22}. A goal of this paper is to provide geometric proofs of some results of \cite{PR24}. 

\begin{theorem}[Theorem \ref{Polishchuk}, Theorem A of \cite{PR24}]
Let $\partial, n$ be coprime integers with $n>0$. Then there is an exceptional collection $E,\O_Y$ on del Pezzo surfaces $Y$ of degree $4$  with $\rk(E)=n$ and $\deg(E)=\partial$.  
\label{mainthm8}
\end{theorem}

%A goal of this paper is to provide geometric proofs of certain results in \cite{PR24}. In particular, by combining Theorems \ref{mainthm4} and \ref{mainthm7}, we give two geometric proofs of the main theorem of \cite{PR24}.

In \cite{PR24} the authors showed that for a del Pezzo surface $Y$ of degree $\ell\geq 5$, not every rank and degree are possible for a vector bundle $E$ in an exceptional collection $E,\O_Y$. Using our characterization of Wahl singularities occurring in a degeneration (Theorem~\ref{mainthm1}), and the fact that any exceptional vector bundle on a del Pezzo surface is a Hacking exceptional bundle \cite{H12,Per18} (see \ref{HecfordelPezzo}), we prove the following more general result.

\begin{theorem}[Theorem \ref{strongerPolishcuk}]
There exist infinitely many pairs $\partial,n$ of coprime integers with $0<\partial<n$, such that an e.v.b. $E$ with $\rk(E)=n$ and $\deg(E)\equiv \pm\partial \pmod{n}$ is not realizable on any del Pezzo of degree $\geq 5$.
\label{mainthm9}
\end{theorem}

Many concrete examples of exceptional collections $E,\O_Y$ can be worked out in degrees $\ell \geq 5$ via $W_{*m}$ of type (I), see Remark \ref{examplesPolishchuk}. In Remark \ref{AllPolishchuk}, we show a geometric realization of all $E,\O_Y$ from \cite[Proposition 5.2]{PR24}.

\subsubsection*{Acknowledgments} We thank Jenia Tevelev for drawing our attention to \cite{PR24}, and Jonny Evans for engaging discussions on this topic. Preliminary results emerged during the V ELGA, held in Cabo Frio, Brazil, in August 2024. The authors are grateful for the excellent environment provided by this CIMPA school. The first author was supported by FONDECYT Regular Grant 1230065. The second author was supported by the ANID National Doctoral Scholarship 2022–21221224.

\tableofcontents

\subsubsection*{Notation}

\begin{itemize}
    \item[] $W$ is a projective surface with only Wahl singularities.
    \item[] $W_t \rightsquigarrow W$ denotes a one-parameter $\Q$-Gorenstein smoothing $W_t$ of $W$.
    \item[] $W_*$ is a projective surface with only one Wahl singularity.
    \item[] $W_{*m}$ is the surface corresponding to a del Pezzo Wahl chain (Definition \ref{def: markedsurface}).
    \item[] $W_{*m}^T$ is the toric del Pezzo surface corresponding to $W_{*m}$.
    \item[] $\widehat{W}_{*m}$ is the (toric) M-resolution of $W_{*m}^T$.
    %\item[] $W_{*mm}$ is the surface corresponding to a minimally marked Wahl singularity.    
\end{itemize}

%----------------------------------------------------------------------------
\section{Hirzebruch-Jung continued fractions and Wahl chains} \label{s1}

More details and proofs could be found, for example, in \cite{UZ23}.

\begin{definition}
A collection $\{e_1,\ldots,e_r \}$ of positive integers admits a \textit{Hirzebruch-Jung continued fraction} (HJ continued fraction) $$[e_1,\ldots, e_r]:=  e_1 - \frac{1}{e_2 - \frac{1}{\ddots - \frac{1}{e_r}}},$$ if $[e_i,\ldots,e_r] >0$ for all $i\geq 2$, and $[e_1,\ldots, e_r]\geq 0$. Its \textit{value} is the rational number $[e_1,\ldots, e_r]$, and its \textit{length} is $r$. %The bracket $[e_1,\ldots, e_r]$ is called HJ continued fraction.
\label{hjcf}
\end{definition}

If the HJ continued fraction $[e_1,\ldots, e_r]$ has $e_i \geq 2$ for all $i$, then its value is a rational number $\frac{m}{q}>1$. This gives a one-to-one correspondence between $\Q_{>1}$ and those HJ continued fractions. For $0<q<m$ coprime integers we refer to it as the HJ continued fraction of $\frac{m}{q}$. On the other hand, if an HJ continued fraction contains $1$s, then it "blows-down" to a unique minimal HJ continued fraction. They are $[1,1]$, $[1]$, or are represented by some $\frac{m}{q} \in \Q_{>1}$.

\begin{lemma}
Given an HJ continued fraction $[\ldots,u,1,v,\ldots]$ that is not $[1,1]$ or $[1]$, we have the blow-down HJ continued fraction $[\ldots,u-1,v-1,\ldots]$, and conversely, given an HJ continued fraction $[\ldots,u,v,\ldots]$, the blow-up $[\ldots,u+1,1,v+1,\ldots]$ is an HJ continued fraction. (This includes the cases $[1,v,\ldots]$ and $[\ldots,u,1]$.) Moreover, we have equality on values $$[\ldots,u,1,v,\ldots]=[\ldots,u-1,v-1,\ldots]$$ when the blow-down $1$ is not at the first position.
\label{updown}
\end{lemma}

\begin{proposition}
Let $0<q<m$ be coprime, and let $\frac{m}{q}=[e_1,\ldots, e_r]$ and $\frac{m}{m-q}=[b_1,\ldots,b_s]$ be the corresponding HJ continued fractions. Then $[e_1,\ldots,e_r,1,b_s,\ldots,b_1]=0$, i.e. the minimal model for $[e_1,\ldots,e_r,1,b_s,\ldots,b_1]$ is $[1,1]$. We call $[b_1,\ldots,b_s]$ the \textit{dual} of $[e_1,\ldots,e_r]$. 
\label{dual}
\end{proposition}

\begin{definition}
A \textit{Wahl chain} is the HJ continued fraction of $\frac{n^2}{na-1}$ for some $0<a<n$ coprime. 
\label{wahl}
\end{definition}

Every Wahl chain can be obtained via the following algorithm due to J. Wahl (see \cite[Prop.3.11]{KSB88}): 

\begin{itemize}
    \item[(i)] $[4]$ is the Wahl chain for $n=2$ and $a=1$.
    \item[(ii)] If $[e_1,\ldots,e_r]$ is a Wahl chain, then $[e_1+1,e_2,\ldots,e_r,2]$ and $[2,e_1,\ldots,e_{r-1},e_r+1]$ are Wahl chains.
    \item[(iii)] Every Wahl chain is obtained by starting with (i) and iterating the steps in (ii).
\end{itemize}

Under this algorithm, the new position of the initial $[4]$ is called the \textit{center} of the Wahl chain. For example, the Wahl chain $[3,2,2,2,2,6,2,7,2]$ has its center at $6$. 

\begin{proposition}
If $[e_1,\ldots,e_r]$ is a Wahl chain, then $\sum_{i=1}^r e_i=3r+1$.
\label{sumWahl}    
\end{proposition}

\begin{proposition}
Let $\Wa$ be a Wahl chain for some $0<a<n$. Then its dual chain corresponds to $\frac{n^2}{n^2-na+1}=[x_1,\ldots,x_u,2,y_v,\ldots,y_1]$ where $\frac{n}{a}=[y_1,\ldots,y_v]$ and $\frac{n}{n-a}=[x_1,\ldots,x_u]$. In fact, we have $\Wa=[y_1,\ldots,y_{v-1},y_v+x_u,x_{u-1},\ldots,x_1]$.
\end{proposition}

\begin{definition}
An HJ continued fraction is a \textit{zero continued fraction} if its minimal model is $[1,1]$.
\label{zeroHJ}
\end{definition}

Zero continued fractions are recovered via "blowing-up": $[1,1]$, $[1,2,1]$, $[2,1,2]$, $[1,2,2,1]$, $[2,1,3,1]$, $[1,3,1,2]$, $[3,1,2,2]$, $[2,2,1,3]$, etc. For a given length $s$, there are $\frac{1}{s}\binom{2(s-1)}{s-1}$ zero continued fractions. They are in one-to-one correspondence with triangulations of convex polygons with $s+1$ sides. It is relevant to us to know when an HJ continued fraction admits a zero continued fraction; see Definition \ref{zcfSigma}. The next proposition is well-known.

%\begin{definition}
%We say that a HJ continued fraction $[e_1,\ldots,e_r]$ with $e_i \geq 2$ admits a \textit{zero continued fraction} of weight $\lambda$
%if there are indices $i_1<i_2<\ldots<i_{\mu}$ for some $\mu \geq 1$ and integers $d_{i_k}\geq 1$ such that $$[\ldots,e_{i_1}-d_{i_1},\ldots,e_{i_2}-d_{i_2},\ldots,e_{i_{\mu}}-d_{i_{\mu}},\ldots]=0,$$  and $\lambda+1=\sum_{k=1}^{\mu} d_{i_k}$.
%\label{zcfSigma}
%\end{definition}

\begin{proposition}
The HJ continued fraction $[e_1,\ldots,e_r]$ admits a zero continued fraction of weight $0$ if and only if it is the dual of a Wahl chain.
\label{weight0} 
\end{proposition}

Classification for higher weights is more complicated. It has to do with M-resolutions of cyclic quotient singularities (c.q.s.). We will review this in Section \ref{s3}. For example, weight $1$ corresponds precisely to c.q.s. that either admit an extremal P-resolution, or are HJ continued fractions of $\frac{2n^2}{2na-1}$ with gcd$(a,n)=1$.

\begin{proposition}
Consider an HJ-continued fraction $[e_1,\ldots,e_r]$ with $e_j \geq 2$ for all $j$. Take $i>1$, and let $\frac{n}{a}=[e_1,\ldots,e_{i-1}]$. Then $$[e_1,\ldots, e_{i-1}+x_u,x_{u-1},\ldots,x_1,1,e_1,\ldots,e_r]$$ contracts to $[e_1,\ldots, e_{i-1},e_i-1,e_{i+1}, \ldots,e_r]$ where $\frac{n}{n-a}=[x_1,\ldots,x_u]$. 
\label{prop:slide}
\end{proposition}

\begin{proof}
This easily follows from Proposition \ref{dual}.
\end{proof}

Proposition \ref{prop:slide} is the combinatorial content of sliding, which is a tool used in Subsection \ref{s3.4}. 

%----------------------------------------------------------------------------
\section{Del Pezzo Wahl chains} \label{s2}

Del Pezzo Wahl chains were defined in Definition \ref{DPWahlchain}. The Wahl chain $$[2,2,2,10,2,2,2,2,2,5]$$ is del Pezzo for $18$ distinct markings. For example, it is del Pezzo of degree $9$ for the marking $[2,1,2,\underline{10},2,2,2,2,1,5]$. It is also del Pezzo of degree $8$ for the marking $[1,2,1,\underline{10},2,2,2,2,1,5]$. The Wahl chain $[2, 7, 2, 2, 3]$ is del Pezzo of degree $4$ for the markings $[1, 3, 1, 2,\underline{3}]$, $[\underline{2},3, 1, 2, 2]$, and $[\underline{2},2, 2, 1, 3]$. See \href{https://colab.research.google.com/drive/1tooXaHZnug0nCRvmXs9g_WIi1jFm4dgs?usp=sharing}{MarkingsWahlChains} for more examples. In general, a Wahl chain could be del Pezzo of a fixed degree for various markings and different types, but canonically we always have:

\begin{lemma}
Except for $[4]$ and $[2,5]$, every Wahl chain $[e_1,\ldots,e_r]$ is del Pezzo of type (I) and degree $4$ for two canonical markings. The central marks are $e_1$ or $e_r$. 
\label{canmarking}
\end{lemma}
    
\begin{proof}
Let $\Wa$ be a Wahl chain different from $[2,\ldots,2,x+4]$. Let $e_i$ be the center of the Wahl chain. We mark it with $d_i=4$. Let $e_{i-1}$ or $e_{i+1}$ be the position right after we move from the center using the Wahl algorithm, we mark it with $d_{i-1}=1$ or $d_{i+1}=1$ accordingly. Consider $e_1$ as the central mark, and mark $e_r$ with $d_r=1$. That is one marking. For the other, consider $e_r$ as the central mark, and mark $e_1$ with $d_1=1$. Then $[e_2,\ldots,e_{r}]$ or $[e_1,\ldots,e_{r-1}]$ admits a zero continued fraction with that marking, because of the Wahl algorithm. If $\Wa=[2,\ldots,2,x+4]$, then we mark it as $[\underline{2},2,\ldots,2,1,x-1]$, or $[1,2,\ldots,2,1,\underline{x}]$, and this works as degree $4$ when $x\geq 2$. 
\end{proof}

The previous lemma defines \textit{canonical markings} of degree $4$ for any Wahl chain. 

For example, $[2,6,2,3]$ has canonical markings $[1,2,1,\underline{3}]$ and $[\underline{2},2,1,2]$; 

for $[2,2,2,5,5]$ we have $[\underline{2},2,2,1,3]$ and $[1,2,2,1,\underline{5}]$; 

for $[2,2,2,2,8]$ they are $[\underline{2},2,2,1,3]$ and $[1,2,2,1,\underline{8}]$; 

for $[2,2,7,2,2,2,2,2,2,5,9,2,2,2,2,4]$ the canonical markings are $$[1,2,7,2,2,2,2,2,2,1,8,2,2,2,2,\underline{4}] \ \ \text{and} \ \  [\underline{2},2,7,2,2,2,2,2,2,1,8,2,2,2,2,3].$$ 

Of course, there are plenty of other markings. The Wahl chains $[2,\ldots,2,A+4]$ admit the marking $[1,2,\ldots,2,1,\underline{A+4}]$ which is del Pezzo of degree $8$. The particular case $[2,2,2,7]$ admits $[2,1,2,\underline{7}]$ which is del Pezzo degree $9$. There are more markings for $[2,\ldots,2,A+4]$, but when $A\geq 6$ we have the above (degree $8$), the canonical (degree $4$), and some other few families of markings of degree $1$ and $2$. In general, it is hard to list all possible markings, but for degree bigger than or equal to $5$ we at least have a restriction on the central mark.   

\begin{lemma}
Let $[e_1,\ldots,e_r]$ be a del Pezzo Wahl chain of degree $\ell$, and central mark $d$.
\begin{itemize}
\item If the marking is of type (I), then $d=\ell+u-3$ for some $u\geq 0$.
\item If the marking is of type (II), then $d=\ell+u_1+u_2-1$ for some $u_i\geq 0$.
\end{itemize}
\label{restrCentralMark}
\end{lemma}

\begin{proof}
It is not hard to see that every zero continued fraction can be constructed from a unique $[1,1]$, $[1,2,1]$, $\ldots$, $[1,2,\ldots,2,1]$ of length $u+1$ by applying some $v$ blow-ups of the form $[\ldots,a,b,\ldots] \mapsto [\ldots,a+1,1,b+1,\ldots]$. Therefore, if $[k_1,\ldots,k_s]$ is a zero continued fraction, then $\sum_{i=1}^r k_i = 2u+3v.$   

Assume type (I). Let $[k_1,\ldots,k_{r-1},\underline{e_r}]$ be the marking of degree $\ell=9-\lambda$. Then Proposition \ref{sumWahl} and the formula above gives the equality $$\sum_{i=1}^r e_i = 3(u+1+v+1)+1=2u+3v+\lambda+1+d,$$ and so $d=\ell+u-3$.  Similarly for (II) and marking $[k_1,\ldots,k_{i-1},\underline{e_i},k_{i+1},\ldots,k_r]$ with $u_i$s and $v_i$s, we obtain $d=\ell+u_1+u_2-1$.
\end{proof}

\begin{corollary}
Let $[e_1,\ldots,e_r]$ be del Pezzo of degree $\ell  \geq 5$.
\begin{itemize}
\item If the marking is of type (I) and not $[0,\underline{2}]$ (it is degree $5$), then the central mark is $\geq 3$.
\item If the marking is of type (II), then the central mark is $\geq 5$.
\end{itemize}    
\label{delP>=5}
\end{corollary}

\begin{proof}
This is Lemma \ref{restrCentralMark}. If $u=0$ in type (I), then we have only $[0,\underline{2}]$ or $[0,\underline{5}]$.
\end{proof}

\begin{remark}
Let $[e_1,\ldots,e_r]$ be a Wahl chain for some $0<a<n$. Assume it is del Pezzo of type (I) with degree $\ell \geq 5$ and central mark $e_1$. Then, 
for any $k\geq 1$, $$\frac{n_{k+1}^2}{n_{k+1}n_k-1}=[\underbrace{\ell-2,\dots,\ell-2}_{k},e_1,\dots,e_r+1, \underbrace{2,\dots,2,3,\dots,2,\dots,2,3}_{(k-1) \ \text{blocks} \  2,\dots,2,3},\underbrace{2,\dots,2}_{\ell-4}]$$ is a del Pezzo Wahl chain of type (I) and degree $\ell$ with central mark $\ell-2$ and marking induced by the marking on $[e_1,\ldots,e_r]$. Here, the blocks $2,\ldots,2,3$ have $(\ell-5)$ $2$s, and the $n_k$ satisfy $n_{k+1}=(l-2)n_k-n_{k-1}$ for all $k\geq 1$, where $n_0=a$ and $n_1=n$. See explicit Wahl chains in Example \ref{nomarkov} and Remark \ref{AllPolishchuk}. 
\label{Pellrecursion}
\end{remark}

\begin{proposition}
Consider Wahl chains of the form $\Wa=[2,\ldots,2,A+4,2,\ldots,2,B+2]$. If $A \geq 2$ and $B>A+3$, then $\Wa$ is never del Pezzo of degree greater than or equal to $5$.
\label{5isnotfree}
\end{proposition} 

\begin{proof}
By Corollary \ref{delP>=5}, we only need to check the possible markings $[\ldots,\underline{B+2}]$ in type (I) and $[\ldots,\underline{A+4},\ldots]$ in type II. The proof is then a straightforward computation with few possibilities.   
\end{proof}

\begin{remark}
In fact, given a Wahl chain $[e_1,\ldots,e_r]$ distinct from $[4]$ and $B \gg 0$, one can show that $[2,\ldots,2,x_1+1,x_2,\ldots,x_r,B+2]$ is not del Pezzo of degree $\ell \geq 5$. Indeed, if the central mark is not $B+2$, then it must be some $x_i \geq 3$. But to make $B+2$ equal to $1$ we would need to subtract too much as $B \gg 0$. When the central mark is $B+2$, we have a chain of $B$ $2$s that would collapse over some $x_i$, unless the first $2$ is marked. If that happens there must be a marking for $[x_1,\ldots,x_r]$ so that it becomes $[0]$, but this is a Wahl chain not equal to $[4]$, and so we need at least $5$ markings on $[x_1,\ldots,x_r]$, and then the degree is $\leq 4$. Note that for $[2,\ldots,2,5,B+2]$ and any $B$ we have a marking of degree $5$.
\label{muchmore}
\end{remark}

We define the following sets of Wahl chains $$\setW_{\ell}=\{ \text{del Pezzo Wahl chains of degree} \ \ell \}.$$

By Lemma \ref{canmarking} we have $\setW_{4}$ is the set of Wahl chains except for $[4]$ and $[5,2]$. Proposition \ref{5isnotfree} establishes that there are infinitely many Wahl chains which are not in $\setW_{\ell}$ for $\ell=5,6,7,8,9$. 

\begin{remark}
Given a del Pezzo Wahl chain in degree $\ell$, there is a W-blowing-up procedure to construct infinitely many del Pezzo Wahl chains in degree $\ell+1$. We will see this in \ref{s3.2}.   
\end{remark}

A well-known fact is the following theorem; see \cite[Proposition 4.1]{UZ23}. A Markov triple $(n_1,n_2,n)$ is a set of positive integers such that $n_1^2+n_2^2+n^2=3n_1 n_2 n$.  

\begin{theorem}
A Wahl chain $\Wa$ of index $n$ belongs to $\setW_9$ if and only if it is Markov, that is, there is a Markov triple $(n_1,n_2,n)$ with $n_1,n_2<n$ such that $\Wa$ corresponds to the minimal resolution of $\P(n_1^2,n_2^2,n^2)$ over the Wahl singularity of index $n$. A Wahl chain in $\setW_9$ has a unique marking of degree $9$.
\label{markovThm}
\end{theorem}

\begin{proof}
We refer to the proof of \cite[Proposition 4.1]{UZ23}. For the last statement, see \cite[Theorem 3.10]{UZ23}, which is \cite[Proposition 3.15]{Aig13}.
\end{proof}

\begin{conjecture}[Frobenius uniqueness conjecture \cite{Aig13}]
A Wahl chain $\Wa$ in $\setW_9$ is determined by its index. 
\label{ConjMarkov}
\end{conjecture}

\begin{question}
Is there an algorithm to describe $\setW_{\ell}$ for each $\ell$?
\label{markovConj}
\end{question}

In the next section, we will see that this question has to do directly with particular partial resolutions of c.q.s., and so it may be hard to give an answer as the one obtained for $\setW_9$, where one sees an equation that determines everything. In Section \ref{s7}, we will explain $\setW_{\ell}$ for $\ell <9$ with Diophantine equations, similar to the case of the Markov equation.

For $\setW_{\ell}$ with $\ell<9$, we do not have that a del Pezzo Wahl chain determines a marking for a given degree. We also do not have anything similar to Conjecture \ref{markovConj} for degrees $\geq 8$, as the next example shows in degree $8$. 

\begin{example}
The analogue to Markov's Conjecture \ref{ConjMarkov} is completely false for degrees $\ell \geq 8$. For example, for $n=29$, we have markings of degree $8$ for $a=4,5,22$. In fact, for $a=4$ we have $[8,2,2,7,2,2,2,2,2,2]$, with the marking $[\underline{8},1,2,7,1,2,2,2,2,2]$; for $a=22$ we have $[2,2,2,10,2,2,2,2,2,5]$, with the marking $[1,2,1,\underline{10},2,2,2,2,1,5]$. For $a=5$ the Wahl chain is $[6,7,2,2,3,2,2,2,2]$, and there are two markings: $[\underline{6},7,1,2,2,2,2,2,2]$ and $[\underline{6},6,2,1,3,2,2,2,2]$. We can generalize this last example via Remark \ref{Pellrecursion}. %Pell numbers are defined through the recursion $p_{k+1} = 6p_k - p_{k-1}$, where $p_0 = 5$ and $p_{-1} = 1$. We observe that if $\frac{p_k^2}{p_kp_{k-1}-1} = [e_1,\dots,e_r]$, then  $\frac{p_{k+1}^2}{p_{k+1}p_{k}-1} = [6,e_1,\dots,e_r+1,2,2,2,2]$.
This leads to the expression: $\Big[\binom{n_k}{n_{k-1}}\Big] = [\underbrace{6, \dots, 6}_k, 7, 2, 2, 3, \underbrace{2,2,2,3,\dots,2,2,2,3}_{(k-1) \ \text{blocks} \ 2,2,2,3},2,2,2,2]$.
These singularities form del Pezzo Wahl chains of type (I) of degree $8$ for two distinct markings: $$[\underline{6},\dots,6,7,1,2,2,2,2,2,3,\dots,2, 2,2,3,2,2,2,2]$$ and
$[\underline{6},\dots,6,6,2,1,3,2,2,2,3,\dots,2, 2,2,3,2,2,2,2]$. A curiosity: they give marked surfaces $W_{*m}^i$ for $i=0,1$ (Definition \ref{def: markedsurface}) whose smoothings are $\F_0$ and $\F_1$ respectively. There is a wormhole between the corresponding degenerations as defined in \cite{UV22}. For $n=55$ and $a=9$ we also have two markings in degree $8$ that are not related by a wormhole, and smoothings are deformations of $\F_0$.
\label{nomarkov}
\end{example}

%----------------------------------------------------------------------------
\section{Deformations and birational geometry} \label{s3}

This section is a brief recap of the geometry needed for the next sections.

\subsection{Chain of Wahl singularities and partial resolutions} \label{s3.1}

\begin{definition}
Let $0<\Omega<\Delta$ be coprime integers. A \textit{cyclic quotient singularity} (c.q.s.) $\frac{1}{\Delta}(1,\Omega)$ is the surface germ at $(0,0)$ of the quotient of $\C^2$ by $(x,y) \mapsto (\zeta x, \zeta^{\Omega} y)$, where $\zeta$ is a $\Delta$-th primitive root of $1$. 
\label{def:cqs}
\end{definition}

For example, c.q.s. $\frac{1}{\Delta}(1,\Delta-1)$ are the Du Val singularities of type $A_{\Delta-1}$. 

A c.q.s. can be minimality resolved via a chain $E_1,\ldots,E_r$ of rational curves. We have $E_i\simeq \P^1$, $E_i\cdot E_j=0$ if $|i-j|>1$, $E_i\cdot E_{i+1}=1$, and $E_i^2=-e_i\leq -2$ where $[e_1,\ldots,e_r]$ is the HJ continued fraction of $\frac{\Delta}{\Omega}$. On the other hand, given such a chain of rational curves in a projective nonsingular surface $X$, there exists a birational morphism $X \to Y$ that precisely contracts them into $\frac{1}{\Delta}(1,\Omega)$, and $Y$ is a projective surface. That is the Artin's criteria for contractibility of curves on algebraic surfaces \cite{Art62}.   

\begin{definition}
A \textit{T-singularity} is a Du Val singularity or a c.q.s. $\frac{1}{dn^2}(1,dna-1)$, where $d\geq 1$ and $0<a<n$ are coprime integers. A \textit{Wahl singularity} is a c.q.s. $\frac{1}{n^2}(1,na-1)$, where $0<a<n$ are coprime integers. We include smooth points setting $n=1$. We will represent non-Du Val T-singularities by $\left[d{n \choose a}\right]$; for Wahl singularities we write $\left[{n \choose a}\right]$.
\label{def:Tsing}
\end{definition}

\begin{definition}
A \textit{chain of Wahl singularities} is a collection of nonsingular rational curves $\Gamma_1,\ldots,\Gamma_{\lambda}$ and a collection of Wahl singularities $P_0,\ldots,P_{\lambda}$ on a surface $W$ such that $P_i, P_{i+1}$ belong to $\Gamma_{i+1}$, and $\Gamma_i, \Gamma_{i+1}$ form a toric boundary at $P_i$ for all $i$. In the case of $i=0$ or $i=\lambda$, we have only one part of the toric boundary. The notation is $P_i=\frac{1}{n_i^2}(1,n_i a_i -1)$, where the minimal resolution goes from $\Gamma_i$ to $\Gamma_{i+1}$. In the minimal resolution of all singularities, the proper transforms of $\Gamma_i$ have self-intersection $-c_i$. This situation in $W$ will be denoted by $$\left[{n_0 \choose a_0}\right]-(c_1)-\left[{n_1 \choose a_1}\right]-(c_2)-\ldots -(c_{\lambda})- \left[{n_{\lambda} \choose a_{\lambda}}\right].$$ When $P_i$ is smooth (i.e. $n_i=1$), then we write just $\ldots-(c_{i})-(c_{i+1})-\ldots$. For a chain of Wahl singularities, we define $\delta_i:=n_{i-1}n_{i} |\Gamma_i \cdot K_W| \in \Z_{\geq 0}$.
\label{def:chainWahlsing}
\end{definition}

\begin{example}
Consider toric degenerations $W$ of $\P^2$. They are in correspondence with Markov triples $(a,b,c)$. Take $W=\P^2(a^2,b^2,c^2)$. If $(1<a<b<c)$, then $W$ has the chain of Wahl singularities 
$$\left[{a \choose w_a}\right]-(1)-\left[{c \choose w_c}\right]-(1)- \left[{b \choose w_b}\right].$$ If $(a=1<b<c)$, then $(0)-\left[{c \choose w_c}\right]-(1)- \left[{b \choose w_b}\right]$. If $a=b=1<2=c$, then $(0)-\left[{2 \choose 1}\right]-(0)$. The $w_x$ are computed via the Markov equation. 
\label{ex:Markov}
\end{example}

\begin{definition}
A chain of Wahl singularities is \textit{contractible} if there exists a contraction of the chain into some c.q.s. $\frac{1}{\Delta}(1,\Omega)$. We represent it as $$\left[{n_0 \choose a_0}\right]-(c_1)-\left[{n_1 \choose a_1}\right]-(c_2)-\ldots -(c_{\lambda})- \left[{n_{\lambda} \choose a_{\lambda}}\right] \to \frac{1}{\Delta}(1,\Omega).$$
\label{def:contractChainWahlsing}
\end{definition}

Contractible chains of Wahl singularities can be thought as particular partial resolutions of $\frac{1}{\Delta}(1,\Omega)$. The minimal resolution of $\frac{1}{\Delta}(1,\Omega)$ is a contractible chain of Wahl singularities, where all points $P_i$ are nonsingular. Example \ref{ex:Markov} is not contractible. If $\frac{\Delta}{\Omega}=[e_1,\ldots,e_r]$, then a contractible chain of Wahl singularities can be seen as an HJ continued fraction whose minimal model is $[e_1,\ldots,e_r]$ together with a choice of subWahl chains connected by $c_i$s. 

For a given $\frac{1}{\Delta}(1,\Omega)$, there are finitely many choices of chains of Wahl singularities that are responsible for all possible deformations of $\frac{1}{\Delta}(1,\Omega)$. This is the work of Koll\'ar and Shepherd-Barron \cite{KSB88} using M-resolutions \cite{BC94}. %The negative counterpart N-resolutions were defined in \cite{TU22}.

\begin{definition}
Given $0<\Omega<\Delta$ coprime, a \textit{M-resolution} of $\frac{1}{\Delta}(1,\Omega)$ is a contractible chain of Wahl singularities on a surface $W^+$ with data $\left\{
P_i=\left[{n_i \choose a_i}\right] \right\}_{i=0}^{\lambda}$, $\{\Gamma_i\}_{i=1}^{\lambda}$ such that $K_{W^+} \cdot \Gamma_i \geq 0$ for all $i$. %The \textit{N-resolution} of the M-resolution $W^+$ is a contractible chain of Wahl singularities in a surface $W^-$ with data $\left\{\bar P_i=\left[{\bar n_i \choose \bar a_i}\right] \right\}_{i=0}^{\lambda}$, $\{\bar \Gamma_i\}_{i=1}^{\lambda}$ such that 
%\begin{itemize}
%\item[(1)] The singularity $\bar P_{\lambda}$ is the same as $P_0$. Moreover, for every $i=1,\ldots,\lambda$, the contraction of the chain $\bar \Gamma_{\lambda-i+1}\cup\ldots\cup \bar \Gamma_{\lambda} \subset W^-$ is the same c.q.s. as the contraction of the chain $\Gamma_1\cup\ldots\cup\Gamma_i \subset W^+$. %We denote that c.q.s. by ${1\over \Delta_i}(1,\Omega_i)$.  
%\item[(2)] $\bar \delta_{\lambda-i+1}=\delta_i$ for $i=1,\ldots,\lambda$.
%\item[(3)] $K_{W^-} \cdot \bar \Gamma_i \leq 0$.
%\end{itemize}
\label{def:mresnres}
\end{definition}

\begin{remark}
In \cite{KSB88}, the authors instead used P-resolutions, which are exactly the contraction of all curves $\Gamma_i$ with $\delta_i=0$ in a given M-resolution. The resulting surface has only T-singularities. M-resolutions appeared in \cite{BC94}. %In \cite{TU22}, it is proved that given an M-resolution there is a unique N-resolution, and it is constructed in two different ways. 
\end{remark}

\begin{example}
The c.q.s. $\frac{1}{19}(1,7)$ admits three M-resolutions \cite[Ex. 3.15]{KSB88}, where the first M-resolution is the minimal resolution: $$(3)-(4)-(2) \ \ \ \ \ \ \ \ \ \ \ \ \ \ \ \ \ \left[{2 \choose 1}\right]-(1)-\left[{3 \choose 1}\right] \ \ \ \ \ \ \ \ \ \ \ \ \ \ \ \ \ \ (3)-\left[{2 \choose 1}\right]-(2).$$ %The corresponding N-resolutions are: {\small $$\left[{ 8 \choose 3}\right]-(1)-\left[{ 8 \choose 3}\right]-(1)-\left[{2 \choose 1}\right]-(1) \ \  \left[{ 5 \choose 2}\right]-(1)-\left[{ 2 \choose 1}\right] \ \  \left[{ 8 \choose 3}\right]-(1)-\left[{5 \choose 2}\right]-(1).$$} 
\label{ex:ksb}
\end{example}

There is a well-known geometric Koll\'ar--Shepherd-Barron correspondence between the irreducible components of the deformation space of $\frac{1}{\Delta}(1,\Omega)$ and P-resolutions \cite{KSB88}. Christophersen \cite{C91} and Stevens \cite{S91} gave a combinatorial way to find all P-resolutions. If $\frac{\Delta}{\Delta-\Omega}=[b_1,\ldots,b_s]$, they proved that P-resolutions are in bijection with the set $$ K(\Delta/\Delta-\Omega) := \{ [k_1,\ldots,k_s]=0 \
\text{such that} \ 1 \leq k_i \leq b_i \}.$$ As it was said above, P-resolutions are in bijection with M-resolutions. In \cite[Cor.10.1]{PPSU18} a geometric procedure (MMP) is worked out to find the zero continued fraction of an M-resolution.

\subsection{W-surfaces and their MMP} \label{s3.2}
In the previous subsection we considered surfaces with Wahl singularities. In this section, we consider $\Q$-Gorenstein smoothings of them in a systematic way. %For that we now make the next definition \cite{Urz16a}. 

\begin{definition}[\cite{Urz16a}]
A \textit{W-surface} is a normal projective surface $W$ together with a proper deformation $(W \subset \W) \to (0 \in \D)$ such that
\begin{enumerate}
\item $W$ has at most Wahl singularities;
\item $\W$ is a normal complex $3$-fold with $K_{\W}$ $\Q$-Cartier;
\item the fiber $W_0$ is reduced and isomorphic to $W$;
\item the fiber $W_t$ is nonsingular for $t\neq 0$.
\end{enumerate}
We denote this by $W_t \rightsquigarrow W_0:=W$. %The W-surface is said to be \textit{smooth} if $W$ is nonsingular.
\label{wsurf}
\end{definition}

Recall that $\D$ is an arbitrarily small disk, the germ of a nonsingular point on a curve. For a given W-surface, the invariants $q(W_t):=h^1(\O_{W_t})$, $p_g(W_t):=h^2(\O_{W_t})$, $K_{W_t}^2$, $\chi_{\text{top}}(W_t)$ (topological Euler characteristic) remain constant for every $t \in \D$ (see for example \cite[\S 1]{Ma91}).

A W-surface is \textit{minimal} if $K_W$ is nef, and so $K_{W_t}$ is nef for all $t$ (see for example \cite{Urz16a}).  When $K_W$ is nef and big, the \textit{canonical model} of $(W \subset \W) \to (0 \in \D)$ is a $\Q$-Gorenstein deformation of a surface with only T-singularities. (Cf. \cite[Section 2]{Urz16a} and \cite[Sections 2 and 3]{Urz16b}.)

If a W-surface is not minimal, then we can explicitly run the MMP for $K_{\W}$ relative to $\D$, which is fully worked out in \cite{HTU17}. See \cite[\S 2]{Urz16a} for a summary of the results in \cite{HTU17}, and \cite[\S 2]{Urz16b} for details of how MMP is run. (MMP is closed for W-surfaces as shown in the proof of \cite[Theorem 5.3]{HTU17}.) It stops at either a minimal model or a nonsingular deformation of ruled surfaces or a degeneration of $\P^2$ with only quotient singularities (see \cite[Section 2]{Urz16a}). 

There are two operations when we run MMP for W-surfaces: flips and divisorial contractions. In the notation of Koll\'ar and Mori \cite{KM92,M02}, it turns out that they are of type \eni ~or \enii ~only. This is explained in \cite{Urz16b}; we follow the notation in \cite{UZ23}. Let $W_t \rightsquigarrow W$ be a W-surface with $K_W$ not nef. Then there is $\Gamma^- \subset W$ such that $\Gamma^-=\P^1$, $K_{W} \cdot \Gamma^- <0$, $\Gamma^- \cdot \Gamma^- <0$, and $\Gamma^-$ contracts to some c.q.s. $\frac{1}{\Delta}(1,\Omega)$. In the minimal resolution of $W$, the proper transform of $\Gamma^-$ is a $(-1)$-curve. Following the notation in Definition \ref{def:chainWahlsing}, we define $\delta:=-n \, K_W \cdot \Gamma^-$, where $n$ is the product of the indices of the singularities in $W$ contained in $\Gamma^-$. 

\begin{itemize}
    \item For \eni, we have that $\Gamma^-$ is passing through one Wahl singularity. In its minimal resolution, the proper transform intersects one exceptional curve transversally and at one point, say $E_i$. The notation is $$[e_1,\ldots,\overline{e_i},\ldots,e_r].$$ We have $\frac{\Delta}{\Omega} = [e_1,\ldots,e_i-1,\dots,e_r]$.

    \item For \enii, we have that $\Gamma^-$ is passing through two Wahl singularities. In their minimal resolution, the proper transform intersects the ends of the exceptional divisors, each transversally and at one point, say at $E_{r_0}$ and at $F_1$. Let $[e_{1},\ldots,e_{r_0}]$ and $[f_1,\ldots,f_{r_1}]$ be the corresponding Wahl chains, then the notation is  $$[e_{1},\ldots,e_{r_0}]-[f_1,\ldots,f_{r_1}],$$ and $\frac{\Delta}{\Omega} = [e_{1},\ldots,e_{r_0},1,f_1,\ldots,f_{r_1}]$.  
\end{itemize}

\vspace{0.3cm}

\textit{When do we have a divisorial contraction or a flip?} The criterion for any \eni \ or \ \enii \ extremal neighborhood uses the Mori recursion \cite{M02} (see \cite{HTU17,Urz16a} for details).

In the case of a flip, we obtain a new W-surface $W_t \rightsquigarrow W^+$, and a curve $\Gamma^+ \subset W^+$ such that $\Gamma^+=\P^1$, $K_{W^+} \cdot \Gamma^+ >0$, $\Gamma^+ \cdot \Gamma^+ <0$, and $\Gamma^+$ contracts to the same c.q.s. $\frac{1}{\Delta}(1,\Omega)$. If $n^+$ is the product of the indices of the Wahl singularities in $\Gamma^+$, then $K_{W^+} \cdot \Gamma^+=\frac{\delta}{n^+}$. If $W^+ \to W'$ is the contraction of $\Gamma^+$, then over $\frac{1}{\Delta}(1,\Omega) \in W'$ we have an extremal P-resolution \cite[\S 4]{HTU17}.

\begin{definition}
An \textit{extremal P-resolution} $f^+ \colon W^+ \to W'$ of a $\frac{1}{\Delta}(1,\Omega) \in W'$ is a partial resolution with only Wahl singularities such that ${f^+}^{-1}(P)$ is a nonsingular rational curve $\Gamma^+$, and $\Gamma^+ \cdot K_{W^+}>0$. Thus, $W^+$ has at most two singularities (\cite[Lemma 3.14]{KSB88}). 
\label{extremalPres}
\end{definition}

An extremal P-resolution is an example of a contractible chain of Wahl singularities. The following is a flip that we will use frequently. 

\begin{proposition}
Let $[e_1,\ldots,e_{r-1},\overline{e_r}]$ be a flipping \eni. Let $i \in \{1,\ldots,r\}$ be such that $e_i \geq 3$ and $e_j=2$ for all $j>i$. (If $e_r>2$, then we set $i=r$.)

Then the extremal P-resolution for the flip is $(e_1)-[e_2,\ldots,e_i-1]$.
\label{specialFlip}
\end{proposition}

\begin{proof}
    This is \cite[Proposition 2.15]{Urz16b}. 
\end{proof}

In the case of a divisorial contraction, we obtain a new W-surface $W'_t \rightsquigarrow W'$, where for $t\neq 0$ we have the contraction of a $(-1)$-curve $W_t \to W'_t$, and $\Gamma^-$ contracts into a Wahl singularity in $W'$. 

One may wonder if there are antiflips and/or antidivisorial contractions for a given $W_t \rightsquigarrow W$. For flips it is not true in general, as the local deformation over the c.q.s. $\frac{1}{\Delta}(1,\Omega)$ must satisfy (1) and (2) in \cite[Corollary 3.23]{HTU17}. In the case of a divisorial contraction, we have that its associated universal family of divisorial contractions has a one-dimensional base which corresponds to the one-dimensional $\Q$-Gorenstein smoothing of a Wahl singularity \cite[Corollary 3.13]{HTU17}. 

Given a W-surface $W_t \rightsquigarrow W$ and a $(-1)$-curve in $W$ not containing any singularity of $W$, then there is a $(-1)$-curve in $W_t$ for every $t$ forming a contractible divisor in the family \cite[\S IV (4.1)]{BHPV04} (due to Kodaira and Iitaka). Divisorial contractions over a Wahl singularity can be seen as a generalization of this. 

\begin{definition}
Let $W'_t \rightsquigarrow W'$ be a W-surface. A \textit{W-blow-up (down)} is the W-surface $W_t \rightsquigarrow W$ corresponding to a \eni~ or \enii~ divisorial contraction over a Wahl singularity in $W'$. When this happens over a smooth point in $W'$, then we call it \textit{Iitaka-Kodaira blow-up (down)}.
\label{divcontr}
\end{definition}

See \cite{Urz16a,Urz16b} for examples where MMP is run explicitly.

%\vspace{0.3cm} 

In \cite[Cor.10.1]{PPSU18} we describe a geometric algorithm to recover the zero continued fractions from the M-resolution using MMP. The idea is to start with a Hirzebruch surface and a fiber of its $\P^1$-fibration. Over that fiber construct the M-resolution and the dual chain of the c.q.s.: $$\left[{n_0 \choose a_0}\right]-(c_1)-\left[{n_1 \choose a_1}\right]-(c_2)-\ldots -(c_{\lambda})- \left[{n_{\lambda} \choose a_{\lambda}}\right]-(1)-[b_s,\ldots,b_1].$$ Consider $\Q$-Gorenstein smoothings of the corresponding surface keeping the configuration of curves $[b_s,\ldots,b_1]$. Then, we run MMP from the bottom up (left to right), starting with the \eni ~ $\left[{n_{\lambda} \choose a_{\lambda}}\right]-(1)$. At some point, we will find a $(-1)$-curve that does not pass through singularities. Then we have the situation of an Iitaka-Kodaira blow-down. The lifting of the $(-1)$-curve will intersect exactly one of the components of $[b_1,\ldots,b_s]$. These $(-1)$-curves do the markings for each of the $d_{i_j}$. %It turns out that this MMP procedure works for the N-resolution as well, starting with the analogue situation: $$\left[{\bar n_0 \choose \bar a_0}\right]-(\bar c_1)-\left[{ \bar n_1 \choose \bar a_1}\right]-(\bar c_2)-\ldots -(\bar c_{\lambda})- \left[{\bar n_{\lambda} \choose \bar a_{\lambda}}\right]-(1)-[b_s,\ldots,b_1].$$

\subsection{Mori trains} \label{s3.3}

Flips and divisorial contractions for W-surfaces occur in universal families over cyclic quotient singularities \cite{HTU17}. For flips, the c.q.s. must admit an extremal P-resolution. For divisorial contractions, the c.q.s. is a Wahl singularity. The universal family is formed by deformations of \eni ~and \enii ~over the same c.q.s. More precisely an \eni ~degenerates into two \enii~ over $\P^1$, each of them keeps the Wahl singularity in \eni ~and creates a new singularity \cite[\S 2.3]{HTU17}. This is summarized in the next definition. We use the notation of \eni ~and \enii ~via Wahl chains.   

\begin{definition}
Let $(P \in \bW)=\frac{1}{\Delta}(1,\Omega)$ be a cyclic quotient singularity which is the contraction of $\Gamma^-$ for some \eni ~or \enii ~with $\delta>1$. A \textit{Mori train} is the combinatorial data to construct all \eni \ and \ \enii \ over $\frac{1}{\Delta}(1,\Omega)$ of divisorial or flip type. We explain each case separately.
\noindent 

\begin{itemize}
\item[(DC):] Fix the Wahl singularity $\frac{1}{\Delta}(1,\Omega)=\frac{1}{\delta^2}(1,\delta a-1)$. Then its (unique) \textit{Mori train} is the concatenated data of the Wahl chains involved in all \eni \ and \enii \ of divisorial contraction type over $\frac{1}{\delta^2}(1,\delta a-1)$. The \textit{first wagon} corresponds to the Wahl chain $[e_1,\ldots,e_r]$ of $\frac{1}{\delta^2}(1,\delta a-1)$, the next wagons correspond to the Wahl chain in the k1A (and so they have one bar somewhere), and two consecutive wagons correspond to the Wahl chains in the k2A: 
$$[e_1,\ldots,e_r]-[e_{1,1},\ldots,e_{r_1,1}]-[e_{1,2},\ldots,e_{r_2,2}]-\ldots $$

\item[(F):] Fix an extremal P-resolution of $\frac{1}{\Delta}(1,\Omega)$. Its (at most two) \text{Mori trains} are the concatenate data of the Wahl chains involved in all \eni \ and \enii \ of flipping type over $\frac{1}{\Delta}(1,\Omega)$. The \textit{first wagon} corresponds to one of the Wahl chains in the extremal P-resolution, and as before, the next wagons correspond to the Wahl chain in a \eni, and two consecutive wagons correspond to the Wahl chains in a \enii. We put an empty wagon $[]$ if the Wahl singularity is a smooth point. When Wahl singularities in the extremal P-resolution are equal, we have only one Mori train.    
\end{itemize}
\label{moritrain}
\end{definition}

\begin{example} (Flipping family) Let $\frac{1}{11}(1,3)$ be the c.q.s. $(P \in \bW)$. So $\Delta=11$ and $\Omega=3$. Consider the extremal P-resolution $W^+ \to \bW$ defined by $[4]-(3)$. Here $\delta=3$, and $\Gamma^+$ is a $(-3)$-curve (after minimally resolving). Then the Mori trains are
$$[]-[\bar{2},5,3]-[2,3,\bar{2},2,7,3]-[2,3,2,2,2,\bar{2},5,7,3]-\cdots
$$ and
$$[4]-[2,\bar{2},5,4]-[2,2,3,\bar{2},2,7,4]-[2,2,3,2,2,2,\bar{2},5,7,4]-\cdots$$ The initial \enii \ are $[]-[\bar{2},5,3]$ and $[4]-[2,\bar{2},5,4]$, corresponding to the smooth point and the Wahl singularity $\frac{1}{4}(1,1)$ in the extremal P-resolution. Particular examples: we have that $[2,3,\bar{2},2,7,3]$ and $[2,\bar{2},5,4]$ are \eni~ whose flips have $W^+$ as the central fiber. Or $[2,3,2,2,7,3]-[2,3,2,2,2,2,5,7,3]$ is a \enii ~over $\frac{1}{11}(1,3)$.  
\label{exantiflipfam1}
\end{example}

\begin{example}
Over the Wahl singularity $\frac{1}{4}(1,1)$ we have the Mori train of divisorial contractions: $$[4]-[2,\overline{2},6]-[2,2,2,\overline{2},8]-\ldots ,$$ and the Mori train of flips: $[\bar 2,5]-[2,2,\bar 2,7]-[2,2,2,2,2,\bar 2,9]-\ldots$.
\label{exantiflipfam2}
\end{example}

%--------------------------------
\subsection{Slidings and mutations} \label{s3.4}

\begin{definition}
Let $X$ be a nonsingular surface. Let $[e_1,\ldots,e_r]$ be a Wahl chain in $X$ with corresponding exceptional curves $E_1,\ldots,E_r$. Let $W$ be its contraction. Assume that there is a $(-1)$-curve $\Gamma$ in $X$ intersecting $E_i$ with $i>1$ transversally at one point. A \textit{left slide} of $\Gamma$ is a surface $X'$ together with a Wahl chain $[f_1,\dots,f_{r'}]$, the Wahl chain $[e_1,\ldots,e_r]$, and a $(-1)$-curve $\Gamma'$ in between, so that they form the chain $$[f_1,\ldots,f_{r'},1,e_1,\ldots, e_r]$$ which contracts to $[e_1,\ldots,e_{i-1},e_i-1,e_{i+1},\ldots,e_r]$ and is the contraction of $\Gamma$ in $X$. We denote the contraction of both Wahl chains in $X'$ by $W'$, and the image of $\Gamma'$ in $W'$ and of $\Gamma$ in $W$ again by $\Gamma',\Gamma$. Similarly, we define the \textit{right slide} of $\Gamma$ with the notation $\Gamma''$, $W''$. 
\label{slide}    
\end{definition}

\begin{example}
Take the Wahl chain $[3,2,2,7,2]$. Consider the position $i=2$, i.e. the second $2$ in the chain. Then, in terms of HJ continued fractions, the left slide is given by $[5,2,1,3,2,2,7,2]$, and the right slide by $[3,2,2,7,2,1,3,2,2,2,2,5,7,2]$. Similarly, consider $[2,2,2,7]$ and $i=2$. Then the left slide is $[4,1,2,2,2,7]$, and the right slide is $[2,2,2,7,1,2,2,2,2,2,5,7]$.
\end{example}

\begin{lemma}
The left and right slides at $e_i$ in a Wahl chain $[e_1,\ldots,e_r]$ exist and are unique. Let $\frac{n'}{a'}=[e_1,\ldots,e_{i-1}]$ and $\frac{n''}{a''}=[e_r,\ldots,e_{i+1}]$. Then we have $$n' \Gamma' \cdot K_{W'}= \Gamma \cdot K_W=n'' \Gamma'' \cdot K_{W''}<0,$$ and $n'^2 \Gamma'^2=\Gamma^2=n''^2 {\Gamma''}^2$ is not zero.
\label{slides}
\end{lemma}

\begin{proof}
The existence and uniqueness follow from Proposition \ref{prop:slide} applied in both directions. Indeed, the associated Wahl chain of the left slide of $\Gamma$ is the HJ continued fraction of $\frac{{n^\prime}^2}{n^\prime a^\prime-1}$ and of the right slide is $\frac{{n^{\prime\prime}}^2}{n^{\prime\prime}a^{\prime\prime}-1}$.

The second part will be a consequence of Lemma \ref{defslides}. The self-intersection of $\Gamma$ is not zero since $[e_1,\ldots,e_i-1,\ldots,e_r]$ never contracts to $[1,1]$. Indeed, if it does, then this is a Wahl chain and a dual Wahl chain at the same time, a contradiction.    
\end{proof}

\begin{remark}
In fact, if $\delta=-n \Gamma \cdot K_W \in \Z_{>0}$ in Lemma \ref{slides}, then $$\delta=n'a-a'n=n''(n-a)-na''.$$ We have $n'+n''=\delta n$ and $a'+a''=\delta a$. (This can be shown using the matrix equations in Theorem \ref{thm:diophequat}.)  
    \label{calculo}
\end{remark}

\begin{definition}
Let $X',X''$ be the left and right slides, respectively, for some fixed $i$. Let $W',W''$ be the contraction of the corresponding singularities. We say that there is a \textit{mutation} between $W'$ and $W''$.
\label{mutation}
\end{definition}

\begin{definition}
A mutation with $\Gamma^2<0$ is called \textit{Mori mutation}. Otherwise, we call it \textit{Markov mutation}.
\label{mutation}
\end{definition}

Consider an \eni ~ $\mathcal{W}_1 \to \D$ that degenerates into two \enii ~ $\mathcal{W}_2 \to \D$ over $\P^1$ in its universal family. Then there is a Mori mutation between $W_1$ and $W_2$.

\begin{proposition}
Any Mori mutation is given by two consecutive \enii ~in a some Mori train. 
\end{proposition}

\begin{proof}
We have $\Gamma^2<0$ in $W$ if and only if it is contractible, and it must contract to some c.q.s. We also have $\Gamma \cdot K_W <0$, by Lemma \ref{slides}. They are classified in \cite{HTU17}, and the mutation corresponds to \cite[Section 2.3]{HTU17}.
\end{proof}

As in the case of Mori mutations, Markov mutations also define a \textit{Markov train}. As we will not use it, we left the details to the reader, providing only a single example.

\begin{example}
Consider again the Wahl chain $[2,2,2,7]$ and the position $i=2$. Then the Markov mutation defines the Markov train:
$$[4]-[2,\bar{2},2,7]-[2,2,2,2,\bar{2},5,7]-[2,2,2,2,2,3,\bar{2},2,7,7]-\ldots $$ This is precisely the Fibonacci branch of the Markov tree \cite{UZ23}. 
\end{example}

\begin{remark}
Let $W$ be a toric surface with Wahl singularities only. For a curve $\Gamma$ in the toric boundary we could have $\Gamma \cdot K_W <0$, $\Gamma \cdot K_W =0$, or $\Gamma \cdot K_W >0$. The first case is a situation of a mutation. The second case is an M-resolution of a $d=2$ T-singularity. The third must be an extremal P-resolution. Then, over the toric boundary there is a dynamic through the birational geometry of these curves, in the sense of mutations and antiflips. In Section \ref{s5}, we will find ways to "mutate" canonical del Pezzo toric surfaces into simpler del Pezzo surfaces of the same type.     
\end{remark}

%----------------------------------------------------------------------------
\section{Geometric models for Del Pezzo Wahl chains} \label{s4}

%--- More general del Pezzo surfaces (for example with general c.q.s.) may be understood via W-surfaces but I do not know how!!!

\begin{definition}
A \textit{del Pezzo surface} of degree $\ell$ is a projective surface $Z$ with only T-singularities, $-K_Z$ ample, and $K_Z^2=\ell$. A \textit{W-del Pezzo surface} of degree $\ell$ is a W-surface $W_t \rightsquigarrow W$ with $W_t$ (nonsingular) del Pezzo surfaces for $t \neq 0$ and $K_W^2=\ell$. 
\label{def: del Pezzo}
\end{definition}

\begin{remark}
If a W-surface has special fiber $W$ with $-K_W$ ample, then it must be a W-del Pezzo surface, i.e. $-K_{W_t}$ is ample for all $t$ (see for example \cite[Proposition 1.41]{KM98}). However, W-del Pezzo does not imply that $W$ is a del Pezzo surface. This is comparable to the deformations of $\F_0$ into surfaces $\F_{2e}$ with $e\geq 1$. 
\end{remark}

Given a del Pezzo Wahl chain $[e_1,\ldots,e_r]$ of degree $\ell$ and marking $$[k_1,\ldots,k_{i-1},\underline{e_i},k_{i+1},\ldots,k_r]$$ (Definition \ref{DPWahlchain}), we construct its associated surface $W_{*m}$ as follows. Recall that for type (II) we have weights $\lambda_1$ and $\lambda_2$ such that $\ell=9-\lambda_1-\lambda_2 \geq 1$; for type (I) we have $\ell=9-\lambda$, where $\lambda $ is the weight. See Figure \ref{f0} as a guide.

\begin{itemize}
\item Consider the Hirzebruch surface $\F_{e_i}$.
\item For type (II), choose two fibers of the $\P^1$-fibration $\F_d \to \P^1$ and perform blow-ups along these fibers to construct the chains for $[k_1,\ldots,k_{i-1}]$ and $[k_{i+1},\ldots,k_r]$. For type (I) consider one fiber and perform the analogue blow-ups.
\item Perform extra $\lambda_1+1$ and $\lambda_2+1$ (or just $\lambda+1$) blow-ups at general points of the corresponding marked components. Then the Wahl chain has been constructed.
\item Contract the Wahl chain. This is the surface $W_{*m}$.
\end{itemize}

\begin{definition}
For a given del Pezzo Wahl chain, we refer to $W_{*m}$ as its associated \textit{marked surface}. Its type/degree is the type/degree of the del Pezzo Wahl chain. 
\label{def: markedsurface}
\end{definition}

\begin{remark} 
We note that a $W_{*m}$ is not unique up to isomorphism, since the extra blow-ups at general points may introduce parameters. However, they are all deformation equivalent by moving these points. 
\label{W*mnotunique}  
\end{remark}

\begin{definition}
Let $X \to W_{*m}$ be the minimal resolution of $W_{*m}$ (Figure \ref{f0}). A $(-1)$-curve in $X$ corresponding to the extra blow-ups in the marked components is called \textit{marking curve}. A $(-1)$-curve in $X$ may not be a marking curve. In particular, there may be $(-1)$-curves disjoint to the Wahl chain. We call them \textit{floating curves} \cite{CH17}. In general, for a surface $W$ with only T-singularities, a floating curve is a $\P^1$ in $W$ disjoint from singularities and whose self-intersection is $-1$.  
\label{floating}
\end{definition}

It turns out that marked surfaces of type (I) and degree $1$ can be reduced to type (II), or higher degree marked surfaces. Therefore, in degree $1$, it suffices to consider type (II).

\begin{lemma}
Any $W_{*m}$ of type (I) and degree $1$ with no floating curves is also the marked surface of a del Pezzo Wahl chain type (II) and degree $1$.  
\label{wirdo}
\end{lemma}

\begin{proof}
Let $[k_1,\ldots,k_{r-1},\underline{e_r}]$ be the marking, and let $d=e_r$. As in Lemma \ref{restrCentralMark}, we have $d=u-2$, where we have the initial chain $[1,2,...,2,1](d)$ with $u-1\geq 0$ $2$s. Let $[1,2,\ldots,2,1]$ be represented by the chain of curves $E_u,\ldots,E_1, E_0$ respectively. We recall that to complete the Wahl chain, we only need to blow-up at nodes in this chain, and then on suitable marked $E_{i_j}$s. Let $X'$ be the corresponding surface.

Let $C_{\infty}$ be the negative section in $\F_d$, and $F$ be the general fiber of $\F_d \to \P^1$. Consider the linear system $|C_{\infty}+(u-2)F|$ in $\F_{u-2}$. Then in $X'$ we have a pencil of curves from $|C_{\infty}+dF|$. The corresponding fibers intersect transversally the component $E_{u-2}$. Moreover, there is a $(-1)$-curve intersecting transversally $E_{u-1}$ at one point. Since there are no floating curves, no further blow-ups occur over $E_{u-2}$, except possibly at its intersections with $E_{u-1}$ and $E_{u-3}$.

With this data, we can consider another marking on the same Wahl chain with the same degree equal to $1$, whose center is represented by $E_{u-2}$.   
\end{proof}

\begin{example}
Take the Wahl chain $[4,3,5,3,3,2,2]$. It is del Pezzo of type (I) and degree $1$ for the marking $[2,1,3,2,2,1,\underline{2}]$. It defines $W_{*m}^1$. Following the proof above, $W_{*m}^1$ has a floating curve $C$. We have another marking of type (II) and degree $2$ given by $[2,1,2,\underline{3},2,1,2]$. It defines $W_{*m}^2$. Then one can check that there is $W_{*m}^1 \to W_{*m}^2$ blow-up at one nonsingular point, which is the contraction of $C$.    
\end{example}

\begin{theorem}
Let $W_{*m}$ be a marked surface that is not of type (I) and degree $1$. Then $W_{*m}$ is a del Pezzo surface. Moreover, since there are no local-to-global obstructions to deform $W_{*m}$, there exists a W-del Pezzo surface with special fiber $W=W_{*m}$.
\label{theorem:W*m}
\end{theorem}

\begin{proof}
First, suppose that $W_{*m}$ that is of type (I) and degree $\ell \geq 2$. We use the notation in Lemma \ref{wirdo}. By Lemma \ref{restrCentralMark} we have $d=\ell +u-3$, and thus $d+2=\ell-1+u \geq u+1$. This last inequality will be key to show ampleness for $-K_{W_{*m}}$. Let $\phi \colon X \to W_{*m}$ be the minimal resolution, and $\pi \colon X \to \F_d$ the composition of blow-ups (here $d=e_r$). We have $K_X \sim \pi^*(K_{\F_d}) +E$, for some (nonreduced) effective exceptional divisor $E$. We have $K_{\F_d} \sim -2 C_{\infty} - (d+2)F$ in the notation of Lemma \ref{wirdo}. Let $F'$ be the fiber over which the Wahl chain is constructed.  

\vspace{0.3cm}
\noindent
\textbf{Claim:} Let $\Gamma$ be an exceptional curve of $\pi$. Then the multiplicity of $\Gamma$ in $\pi^*((u+1)F')$ minus the multiplicity of $\Gamma$ in $E$ is at least $1$ for non $(-1)$-curves, otherwise it is at least $0$.

\begin{proof}[Proof of the claim]
First, we check for multiplicities up to $X^\prime$, that is, for $E_u,\ldots,E_1$ as in Lemma \ref{wirdo}. For $ \pi^*((u+1)F^\prime)$ the multiplicities are $u+1$, they are at most $u$, so the claim holds for this case. After that, we only blow-up at nodes, and so the difference between multiplicities is at least $1$. Finally, we blow-up at general points of certain exceptional curves, and so the difference is at least $0$.   
\end{proof}

Since $u+1 \leq d+2$, it follows that $-K_X$ is effective by the previous claim. Moreover, $\phi^*(-K_{W_{*m}})$ is $\Q$-effective with support contained in the union of the exceptional curves of both maps. 

Since $K_{W_{*m}}^2>0$, to verify the ampleness of $-K_{W_{*m}}$, it suffices to check that each $(-1)$-curve over the fiber $F'$ intersects $\phi^*(-K_{W_{*m}})$ positively. Each such curve intersects transversally a single exceptional curve of $\phi$, and Wahl singularities have discrepancies in the interval $(-1,0)$. Thus, by the Nakai-Moishezon ampleness criterion we conclude that $-K_{W_{*m}}$ is ample.

Now suppose that $W_{*m}$ is of type (II). Then Lemma \ref{restrCentralMark} 
gives $d=\ell- 1+u_1+u_2$. Therefore, $d+2=\ell-1 +(u_1+1)+(u_2+1)$, and an analogous proof shows that $\phi^*(-K_{W_{*m}})$ is $\Q$-effective with support contained in the union of the exceptional curves of both maps. We apply the Nakai-Moishezon criterion in the same way.  
\end{proof}

\begin{remark}
A Wahl chain could admit two distinct markings that are represented by the same surface $W_{*m}$. For example, the Wahl chain $\Wa=[5,5,3,2,2,2]$ admits the markings
$[1,1,\underline{3},1, 2, 1]$ and $[2, 1, 3, 1, \underline{2},0]$, which are del Pezzo of degree $1$. It is easy to see that the corresponding surfaces $W_{*m}$ are equal. Thus, the function sending a marked Wahl chain into $W_{*m}$ is not injective. This may be a rare phenomenon that occurs only in low degrees. Hence, when we refer to $W_{*m}$, then we think of a particular Wahl chain and marking.   
\label{NOTunique}
\end{remark}

\begin{corollary}
In degree $4$, any Wahl singularity is realizable as a W-del Pezzo surface. 
\end{corollary}

\begin{proof}
When the Wahl chain is not $[5,2]$ and $[4]$, then this is just Lemma \ref{canmarking}, where we consider its canonical marking. Then we apply Theorem \ref{theorem:W*m}. For $[5,2]$ we may take $\F_5$ and the obvious type (I) degree $8$ marking. After that we take $4$ general blow-ups ($4$ floating curves). For $[4]$, consider $\P^2 \rightsquigarrow \P(1,1,4)$. 
\end{proof}

A given $W_{*m}$ defines, in general, infinitely many toric surfaces with Wahl singularities that are degenerations of $W_{*m}$. (Think about one and apply Markov and/or Mori mutations.) Among them, we will choose a particular one, which is del Pezzo. We will use slides (Definition \ref{slide}) in deformations.

\begin{lemma}
Let $W$ be a surface with only Wahl singularities and $H^2(W,T_W)=0$. Suppose $W$ has a Wahl singularity $P$ together with a curve $\Gamma$ defining a slide. Let $W'$ be the left (or right) slide of $\Gamma$. Then there is a $\Q$-Gorenstein deformation of $W$ into $W'$ which smooths the new singularity appearing in $W'$. 
\label{defslides}
\end{lemma}

\begin{proof}
Let $X,X'$ be the minimal resolutions of $W,W'$. We recall that the contraction of $\Gamma$ in $X$ also is the contraction of $X'$ at various $(-1)$-curves over the exceptional chain of $P$. Since $ H^2(X,T_X(-\log E)) \simeq H^2(W,T_W)$, where $E$ is the exceptional divisor over all Wahl singularities in $W$ \cite[Theorem 2]{LP07}, then $H^2(X,T_X(-\log E))=0$. Moreover, using the adding/deleting criteria for $(-1)$-curves \cite[Proposition 4.3]{PSU13}, then $H^2(X',T_X'(-\log E'))=0$, where $E'$ is the exceptional divisor for $W'$. Therefore $H^2(W',T_{W'})=0$. We have no local-to-global obstruction to deform $W'$. 

We consider a one-parameter $\Q$-Gorenstein deformation of $W'$ so that it is a $\Q$-Gorenstein smoothing for the new Wahl singularity, and keeps the other original Wahl singularities unchanged. We simultaneously resolve the original Wahl singularities in the family, and so we obtain a $\Q$-Gorenstein smoothing $\widetilde{W}_t \rightsquigarrow \widetilde{W'}$, where $\widetilde{W'}$ is the resolution of the original Wahl singularities in $W'$. 

We now run a guided MMP (as in \ref{s3.2}) on $\widetilde{W}_t \rightsquigarrow \widetilde{W'}$, using only the $(-1)$-curves over $[1,e_1,\ldots,e_r]$, where $[e_1,\ldots,e_r]$ is the Wahl chain over $P$. We have in $\widetilde{W'}$ the Wahl chain $$\left[{n \choose a}\right]-(1)-(e_1)-\ldots -(e_r),$$ for some $n,a$, where the minimal resolution corresponds to $[f_1,\ldots,f_s,1,e_1,\ldots,e_r]$ and contracts into $[e_1,\ldots, e_i-1,\ldots,e_r]$. We first flip, as in Proposition \ref{specialFlip}, the curve $\Gamma'$. We keep flipping (as in Proposition \ref{specialFlip}) the new $(-1)$-curves from $[e_1,\ldots,e_r]$ until the new central fiber is nonsingular.  

The central fiber of the deformation $\widetilde{W}_t \rightsquigarrow W_0=\widetilde{W'}^{+}$ contains the chain $[e_1,\ldots,e_{i-1}+1,1,e_i,\ldots,e_r]$. At this stage, the chain $[e_1,\ldots,e_r]$ in $W_t$ degenerates to $[e_1,\ldots,e_{i-1}+1,1,e_i,\ldots,e_r]$ with the property that the only curve that breaks is $(e_{i-1})$ into $(e_{i-1})-(1)$. This $(-1)$-curve in the chain lifts to the general fiber, and must intersect  the curve corresponding to $e_i$ transversally.

As with $W'$, the surface $W_0$ has no local to global obstructions to deform, and so $W_t$ is deformation equivalent to our original $W$. This produces the desired degeneration of the statement. 

Moreover, $\Gamma$ degenerates into $n' \Gamma'$, since $\Gamma'$ is a generator of the local class group of the new singularity, which has index $n'$. Since $W \rightsquigarrow W^\prime$ is a $\Q$-Gorenstein deformation, we preserve intersection with canonical class, and therefore we obtain the numerical claims of Proposition \ref{prop:slide}. 
\end{proof}

Given a del Pezzo Wahl chain $[e_1,\ldots,e_r]$ of degree $\ell$ and marking $$[k_1,\ldots,k_{i-1},\underline{e_i},k_{i+1},\ldots,k_r]$$ (Definition \ref{DPWahlchain}), we already constructed its associated surface $W_{*m}$. Let $\phi \colon X \to W_{*m}$ be the minimal resolution. For type (I), we have one special fiber $F'$ containing exceptional curves of $\phi$ and $\pi$, and for type (II) we have two $F'_1$ and $F'_2$. We construct a toric del Pezzo surface $W_{*m}^T$ with only T-singularities and its M-resolution $\widehat{W}_{*m}$ with only Wahl singularities through the next steps:

%\blue{Abajo el orden de los slidings sucesivos creo que debe ser de $1$ a $i-1$ y de $r$ a $i+1$, ya lo tengo puesto así. Te pregunto si esta versión del proceso luce mejor.}

%\red{Cuidado aqui. Los slidings van de abajo hacia arriba, todo esta escrito de esa manera, no cambiar. Van de i-1 a 1 y de i+1 a r. Sino todo cambia. Por favor no hacer ese tipo de cambios bruscos, que puede cambiar todo, regresarlo a como estaba.}

\begin{itemize}
    \item For the $(-1)$-curves in the special fiber corresponding to $[e_1,\ldots,e_{i-1}]$, we perform successive left slidings from index $i-1$ to $1$. For the special fiber corresponding to $[e_{i+1},\ldots,e_{r}]$, we perform the successive right slidings from $i+1$ to $r$.
    \item Each slide yields a $\Q$-Gorenstein degeneration into the corresponding new surface as established in Lemma \ref{defslides}.
    \item For $(-1)$-curves intersecting $E_1$ or $E_r$ we proceed as follows. If $a$ is the number of $(-1)$-curves intersecting $E_1$, we replace the corresponding Wahl chain $[f_1,\ldots,f_u]$ with $$[\underbrace{2,\ldots,2}_{a-1},1,f_1,\ldots,f_u].$$ For $E_r$, we make the respective replacement to the right. As before, the slide construction provides a $\Q$-Gorenstein deformation.
    \item We repeat this process until we obtain a toric surface $\widehat{W}_{*m}$. Note that in $\widehat{W}_{*m}$ we may have zero curves for the canonical class. These are possible ending chains of $(-2)$-curves and/or curves between two Wahl singularities that are equal. In the subsequent theorem, we show that $-K_{\widehat{W}_{*m}}$ is nef.
    \item We perform the contraction of all such zero curves for $K_{\widehat{W}_{*m}}$, which gives the toric del Pezzo surface $W_{*m}^T$.  
\end{itemize}

\begin{figure}[htbp]
\centering
\includegraphics[height=5.8cm, width=10.6cm]{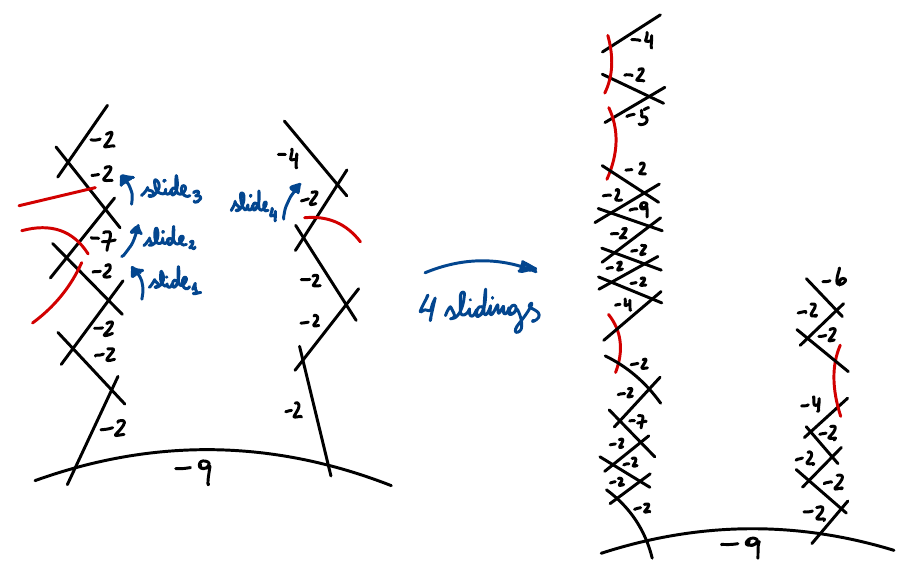}
\caption{The sliding process between the minimal resolutions of $W_{*m}$ and $\widehat{W}_{*m}$, for $n=99$, $a=68$ and a marking with $\ell=7$.} 
\label{f1}
\end{figure}

\begin{theorem}
The surface $W_{*m}^T$ is a toric del Pezzo surface when $W_{*m}$ is not of type (I) and degree $1$. There is a $\Q$-Gorenstein degeneration of $W_{*m}$ into $W_{*m}^T$.
\label{theoremtoric}
\end{theorem}

\begin{proof}
We first prove that $\widehat{W}_{*m}$ satisfies $K_{\widehat{W}_{*m}} \cdot C \leq 0$ for every torus invariant curve $C$. For the curves that arise from the slides, it follows from Lemma \ref{prop:slide}. By performing the first slide on $\Gamma$ on a $E_i$, we obtain $K_{\widehat{W}_{*m}} \cdot \Gamma^\prime<0$. For the remaining of such $\Gamma_j$s intersecting the same $E_i$, we obtain zero curves for the canonical class. Hence, the condition $K_{\widehat{W}_{*m}} \cdot C \leq 0$ holds for curves contained in fibers.

In type (I) of no degree $1$ and in type (II) for any degree, we show that the \say{final} torus-invariant curve $C$, which is the image under $X\to W_{*m}$ of the strict transform of a section in $|C_{\infty} + dF|$, satisfies $K_{\widehat{W}_{*m}} \cdot C < 0$. We have that $\widehat{W}_{*m}$ retains the same $u_i$s of $W_{*m}$ as in Lemma \ref{restrCentralMark}. It follows that in type (I), the intersection is $K_{\widehat{W}_{*m}}\cdot C=-\ell+1-d$ for some $d\in(-1,0]$, while in type (II), we have $K_{\widehat{W}_{*m}}\cdot C=-\ell-1-d_1-d_2$ for some $d_1,d_2\in(-1,0]$.

The degeneration of $W_{*m}$ into $\widehat{W}_{*m}$ is achieved by repeatedly using Lemma \ref{defslides}. It is defined by the slides of each of its marking curves upwards. Now we contract the torus invariant curves $C$ in $\widehat{W}_{*m}$ such that $K_{\widehat{W}_{*m}} \cdot C=0$. Let $\widehat{W}_{*m} \to W_{*m}^{T}$ be the contraction. We then blow-down the deformation \cite[Theorem 1.4]{W76}. We obtain a $\Q$-Gorenstein degeneration of $W_{*m}$ into $W_{*m}^T$.
\end{proof}

\begin{example}
Let $\frac{n^2}{na-1}=[e_1,\ldots,e_r]$ be a Wahl chain that does not have the form $[x+4,2,\ldots,2]$. Let us assume that $e_1>2$ (and so $e_r=2$). By Lemma \ref{canmarking} this Wahl chain has two canonical markings that make it del Pezzo of degree $4$. Let us consider the one with central mark $e_r$. Let $e_i$ be the center of the Wahl chain. This center is marked with $4$ marking curves, and $e_r$ is marked with $1$ marking curve. We also have a marking in $e_{i-1}$ or $e_{i+1}$ depending on how the Wahl algorithm is run on the Wahl chain. If we marked $e_{i+1}$, then we call it right canonical marking; otherwise left. Let us define $$\frac{n_i}{n_i-a_i}=[e_r,\ldots,e_i].$$ Consider the corresponding marked surface $W_{*m}$, and its weak del Pezzo toric surface $\widehat{W}_{*m}$. Then in $\widehat{W}_{*m}$ we have the right chain of Wahl singularities $$(0)-\left[{n \choose a}\right]-(1)-\left[{n_{i+1} \choose a_{i+1}}\right]-(1)-\ldots -(1)- \left[{n_{i+1} \choose a_{i+1}}\right]-(1)-\left[{n_{i+2} \choose a_{i+2}}\right]-(1),$$ or the left chain of Wahl singularities $$(0)-\left[{n \choose a}\right]-(1)-\left[{n_{i} \choose a_{i}}\right]-(1)-\left[{n_{i+1} \choose a_{i+1}}\right]-(1)-\ldots -(1)- \left[{n_{i+1} \choose a_{i+1}}\right]-(1),$$ where the repeated singularities appear four times. The computations follow directly from Proposition \ref{prop:slide} which gives the numerics for an slide. When the Wahl chain has the form $[x+4,2,\ldots,2]$ and $x\geq 2$, then a canonical marking is $[\underline{2},2,\ldots,2,1,x-1]$, and so in $\widehat{W}_{*m}$ we have the chain of Wahl singularities $$(0)-\left[{x+2 \choose x+1}\right]-(1)-\left[{x+4 \choose x+3}\right]-(1)-(2)-(2)-(2)-(2).$$ 
\label{ex:canmarking}
\end{example}

\begin{remark} [Descending mutation] Let us consider $\widehat{W}_{*m}$ for some marked surface $W_{*m}$. We may take a $\Q$-Gorenstein smoothing of the higher index singularity which preserves the toric boundary of $\widehat{W}_{*m}$ and the remaining singularities. This produces an almost toric surface with one sliding curve over some Wahl singularity. Performing the corresponding slide yields a new toric surface with a Wahl singularity of strictly smaller index. Then we have a new toric surface with Wahl singularities. We call it \textit{descending mutation}. In degree $9$, this is exactly the usual mutation between two Markov triples. In general, this gives a way to navigate through toric surfaces with only Wahl singularities and of the same degree until one reaches a smooth toric surface. Details of this process will be worked out in a future work.
\label{treemutation}
\end{remark}

%\begin{definition}
%    --- Define (1) minimal model in general and (2) find $W_{*mm}$. We need to use here the MMP of the previous section, the W-blow-ups.
%\end{definition}

%----------------------------------------------------------------------------
\section{Classification of W-Del Pezzo with one singularity} \label{s5}

Let $W$ be a projective surface with only Wahl singularities,  $K_W^2 \geq 1$, and $-K_W$ big. Let $\phi \colon X \to W$ be the minimal resolution of $W$. Therefore, $X$ is rational. Let $\pi \colon X \to \F_d$ be a composition of blow-downs into a Hirzebruch surface $\F_d$ for some $d>0$. We denote the exceptional divisor of $\phi$ by $\exc(\phi)$. It is a union of Wahl chains. We have the diagram of morphisms $$ \xymatrix{  & X  \ar[ld]_{\pi} \ar[rd]^{\phi} &  \\ \F_d &  & W}$$ We recall that the $\P^1$-bundle $\F_d \to \P^1$ has a unique negative curve $C_{\infty}$ with $C_{\infty}^2=-d$. Let $F$ be the class of a fiber. We have an induced fibration $\pi' \colon X \to \P^1$ via composition of $\pi$ with $\F_d \to \P^1$. 

\begin{proposition}
The surface $W$ has no local-to-global obstructions to deform. For every W-surface $W_t \rightsquigarrow W$ the general fiber is rational, and $h^0(-K_X) \geq K_W^2+1$.
\end{proposition}

\begin{proof}
No local-to-global obstructions follows from \cite[Proposition 3.1]{HP10}.  We now follow \cite[Section 1]{Ma91}. Consider a W-surface $W_t \rightsquigarrow W$, where $W=W_0$. Let $P_n(W_t):=h^0(W_t, {\omega_{W_t}^{\otimes n}}^{\vee \vee})$ for every $t$, where $\omega_{W_t}$ is the canonical sheaf on $W_t$. Then, following the proof of \cite[Theorem 4]{Ma91}, we see that $P_n(W) \geq P_n(W_t)$ for any $n \in \Z$. Since $K_{W_t}^2>0$ is constant on $t$, we have that $-K_{W}$ big implies that $-K_{W_t}$ is big for every $t$. Hence $W_t$ is rational, and by Riemann-Roch on $W_t$ we have $P_{-1}(W_t) \geq K_{W}^2 +1$. On the other hand, by \cite[(3.9.2)]{W81}, we have $h^0(-K_X)=P_{-1}(X)=P_{-1}(W)$, and so $h^0(-K_X) \geq K_W^2 +1$.
\end{proof}

The next definitions appeared in \cite[III.3]{Manetti} to study degenerations of $\P^2$ \cite{Ma91}. 

\begin{definition}
We define $d_{\text{max}}(W)$ as the maximum possible $d$ in the diagram above. A curve $C$ in $X$ is \textit{transversal} to $\pi' \colon X \to \P^1$ if $C \cdot \pi^*(F) >0$.
\label{dmaxtrans}
\end{definition}

We note that the proper transform of $C_{\infty}$ in $X$ has self-intersection $-d_{\text{max}}$ since we cannot blow-up over a point in $C_{\infty}$. We denote it again by $C_{\infty}$, and the pull-back of $F$ again by $F$.

\begin{proposition}
We have $d_{\text{max}}(W)=\text{max} \{-C^2\}$ where $C$ is an irreducible curve in $X$ such that there exists a smooth rational curve $f$ with $f^2=0$ and $f \cdot C=1$.
\label{ultil}
\end{proposition} 

\begin{proof}
    See \cite[Theorem 3.2]{Manetti}.
\end{proof}

\begin{lemma}
Assume that $W$ does not have floating curves. Then the exceptional loci $\exc(\phi)$ can contain at most two transversal curves, and must be sections of $\pi'$. In fact, we have the following options:

\begin{itemize}
    \item[(0)] $\exc(\phi)$ contains no transversal curve.
    \item[(1)] $\exc(\phi)$ contains one transversal curve and it is a section that is a proper transform of a curve in $|C_{\infty}+(d+j)F|$ for some $j=-d,0,1,2$.
    \item[(2)] $\exc(\phi)$ contains two transversal curves. One is $C_{\infty}$ and the other is a section that is a proper transform of a curve in $|C_{\infty}+(d+j)F|$ for some $j=0,1$. In this case, we must have $(\exc(\phi)-C_{\infty}) \cdot C_{\infty} \leq 1$. 
\end{itemize}
\label{keylemma}
\end{lemma}

\begin{proof}
Let $\mathcal{C}$ be a Wahl chain in $\exc(\phi)$. It must have a curve $A$ in $\mathcal{C}$ with $A^2<-3$. Then $h^0(-K_X-A)=h^0(-K_X) \geq K_W^2+1 \geq 2$. Let $N$ be its fixed part of $|-K_X|$. Then $A$ belongs to $N$. As any curve in $\mathcal{C}$ has self-intersection less than or equal to $-2$, we inductively add curves in $\mathcal{C}$ to keep the same number of sections. We finally get $h^0(-K_X-\mathcal{C}) \geq 2$ and $\mathcal{C} \leq N$. As Wahl chains in $\exc(\phi)$ do not intersect among them, we inductively have $h^0(-K_X-\exc(\phi))\geq 2$, and $\exc(\phi) \leq N$. Let $F$ be a fiber of $\pi'$, then $F \cdot (-K_X -\exc(\phi)) = 2-F \cdot \exc(\phi) \geq 0.$ Thus, $\exc(\phi)$ can contain at most two sections, or it contains one (irreducible) double section. However, the subsequent argument easily rules out the possibility of a double section.

In the case $d=d_{\text{max}}(W)=1$ we must have $C_{\infty} \cdot \exc(\phi) >0$. For any $d \geq 1$, if $\exc(\phi)$ does not contain $C_{\infty}$, then $h^0(-K_X-\exc(\phi)-C_{\infty}) \geq 1$ and $F \cdot (-K_X-\exc(\phi)-C_{\infty})=1- F \cdot \exc(\phi) \geq 0$. Thus, $\exc(\phi)$ can contain at most one transversal curve, which must be a section. Hence, by Proposition \ref{ultil}, there is no transversal curve in $\exc(\phi)$ when $d=1$. So, for what follows, we only consider $d \geq 2$.

%Assume that $\Gamma \subset \mathcal{C}$ is transversal. Since $h^0(-K_X-\mathcal{C}) \geq 2$ and $C_{\infty}^2=-d\leq -2$, we have that $h^0(-K_X-\mathcal{C}-C_{\infty}) \geq 1$. But $F \cdot \big(-K_X-\mathcal{C}-C_{\infty} \big) \geq 0$, and so $\Gamma$ must be a section.

Assume that $\exc(\phi)$ contains only one section $\Gamma \neq C_{\infty}$, or a section $\Gamma$ and $C_{\infty}$. We have $h^0(-K_X-\exc(\phi) -C_{\infty}) \geq 1$ or $h^0(-K_X-\exc(\phi)) \geq 2$ respectively. In both cases, $-K_X-\Gamma-C_{\infty}-G$ is linearly equivalent to an effective divisor, where $G$ is formed by fibers or fiber components. Moreover, it follows that $\Gamma \sim C_{\infty} + e F -G'$, where $G'$ is an effective divisor supported on fiber components. Then, $ -K_X-C_{\infty} - \Gamma -G \sim (d+2-e)F+G'-G$. Let us consider a section $C^+ \sim C_{\infty} +dF$. In particular, it avoids intersection with $G'$. Since $(d+2-e)F+G'-G$ is linearly equivalent to an effective divisor and ${C^+}^2=d$, it follows that $C^+ \cdot \big((d+2-e)F+G'-G \big) \geq 0$. This yields to $d+2-e\geq 0$, which restricts the possibilities of $\Gamma$ to those described in (1).

If $\exc(\phi)$ contains a section $\Gamma$ and $C_{\infty}$, then $e \neq d+2$, which rules the possibilities of $\Gamma$ in (2). Moreover, if $(\exc(\phi)-C_{\infty}) \cdot C_{\infty} \geq 2$, then $h^0(-K_X-\exc(\phi)-C_{\infty}) \geq 1$. However, this contradicts that  $F \cdot (-K_X-\exc(\phi)-C_{\infty}) =-1$.
\end{proof}

\begin{remark}
For $d\geq 2$ and $K_W^2\geq 5$, Manetti's \cite[Proposition 9]{Ma91} gives that $\exc(\phi)$ can contain at most $C_{\infty}$ as a transversal curve.   
\label{improvedRemark}
\end{remark}

\begin{theorem}
Assume that $W_*$ is as in the previous lemma and has only one singularity. Consider the diagram $$ \xymatrix{  & X  \ar[ld]_{\pi} \ar[rd]^{\phi} &  \\ \F_d &  & W_*}$$ where $d=d_{\text{max}}(W_*)$. Then we have the following options for $\pi(\exc(\phi))$:
\begin{itemize}
    \item[(0)] One point.
    \item[(1)] One fiber.
    \item[(2)] One section of type $C_{\infty}+(d+i)F$ where $i=-d,0,1,2$.
    \item[(3)] One fiber and one section of type $C_{\infty}+(d+i)F$ where $i=-d,0,1,2$.
    \item[(4)] Two fibers and one section of type $C_{\infty}+(d+i)F$ where $i=-d,0,1,2$.
    \item[(5)] One fiber, $C_{\infty}$ and one section of type $C_{\infty}+dF$.
    %\item[(6)] Two fibers, $C_{\infty}$, and one section of type $C_{\infty}+dF$.
    \item[(6)] One section of type $C_{\infty}+(d+1)F$ and $C_{\infty}$.
\end{itemize}
\label{Thm class}
\end{theorem}

\begin{proof}
By Lemma \ref{keylemma}, we have that $\exc(\phi)$ can contain two sections or one section as transversal curves. If $\exc(\phi)$ does not contain transversal curves, then $\pi(\exc(\phi))$ is a point (0) or one fiber (1). If it contains one section, then we have that $\exc(\phi)$ can contain at most $2$ fibers, because $\exc(\phi)$ is a chain of $\P^1$s. Hence, we have (2), (3) and (4). If $\exc(\phi)$ contains two sections, then one of them must be $C_{\infty}$. We are taking $d=d_{\text{max}}(W)$, and so there are no blow-ups over $C_{\infty}$. Then we have at most $2$ fibers in $\exc(\phi)$. But Lemma \ref{keylemma} (2) implies that $(\exc(\phi)-C_{\infty}) \cdot C_{\infty} \leq 1$, and so the only possible situation is (5) and (6).   
\end{proof}

In the previous theorem, we use $d=d_{\text{max}}$ to control the possible configurations of $\pi(\exc(\phi))$. We now want to prove that for each case in Theorem \ref{Thm class} we have a deformation into $\text{Bl}_u(W_{*m})$ for some $u \geq 0$ and some $W_{*m}$.

As an intermediate step, let us say that $W_*$ is of type (I)' or (II)' if $\pi(\exc(\phi))$ is of type (3) with $i=-d'$ or type (4) with $i=-d'$, where $d'$ is not necessarily $d_{\text{max}}(W_*)$. There may be some floating curves in $W_*$, but we do not blow-up points over $C_{\infty}$. 

\begin{lemma}
A $W_*$ in case (5) of Theorem \ref{Thm class} is of type (II)'. 
\label{caso(5)raro}
\end{lemma}

\begin{proof}
Let us assume that $W_*$ falls into case (5). In $X$ we have a section $\Gamma$ of $X \to \P^1$ from $|C_{\infty}+dF|$ in $\exc(\phi)$. We also have a fiber $F'$ in $\pi(\exc(\phi))$. We denote its image in $\F_d$ by $\Gamma$ again. We must have points in $\Gamma \subset \F_d$ over which we do the blow-ups. As $\Gamma^2=d$ in $\F_d$ and $\Gamma^2 \leq -2$ in $X$, we need at least $d+2$ of these points. Therefore, there is a pencil $\Gamma_t$ in $|C_{\infty}+dF|$ (in $\F_d$) that defines a genus $0$ fibration on $X$. It has at least a section of self-intersection $-e \leq -2$ in $\pi^{-1}(F')$. Hence, we can consider a blow-down into $\F_e$ for $X$. The image of $\exc(\phi)$ must contain at least the new negative section, and two fibers. Therefore, it is of type (4) with $i=-d$ by Theorem \ref{Thm class}, i.e. of type (II)'.    
\end{proof}

A \textit{residual surface} is the contraction of the proper transform of a line in $\P^2$ blown-up at five points. Here, the Wahl chain is $[4]$. 

\begin{lemma}
We can deform a $W_*$ in Theorem \ref{Thm class} to some $W'_*$ of type (I)' or (II)', or a residual surface.
\label{reductiontype}
\end{lemma}

\begin{proof}
Case (0) can be reduced to case (1) through elementary transformations. Of course, the new type (1) is not assumed to have $d=d_{\text{max}}$. 

Let us consider case (1) with $d$ maybe not a $d_{\text{max}}$. First, we have that $\F_d$ is a degeneration of either $\F_1$ or $\F_0$. In either case, we have a deformation of $X_t$ over $\F_i$ with $i=1,0$ into $X$ over $\F_d$, which for each $t$ we have the same configuration $\exc(\phi)$. By blowing down in the family, we obtain a deformation of surfaces $W_t$ with one singularity and of type (1) into $W_*$. So, we are reduced to show that $W_t$ is equivalent to a $W'_*$ of type (I)' or (II)' or the residual case.

Let $F^\prime$ be the fiber over which we construct the Wahl chain. Blowing up a point on $F^\prime$ yields a $(-1)$-curve $C$, and if there is a further blow-up over $C$, we may assume that $C \subset \mathrm{Exc}(\phi)$. Indeed, otherwise, we either deform to a blow-up over $F'$ at two points, or we obtain a floating curve. If there is at least one infinitely near point over $F^\prime$. Then we can easily find a pencil given by all sections passing through that point when $i=1$. When $i=0$ we can just take one ruling. In this way $C$ (or $F'$) becomes a section with $C^2<-1$, reducing to the case that has at least one section. If there is no infinitely near point, then $F'$ is the only member in $\exc(\phi)$, and then we have $[4]$ as the Wahl chain. For $i=0$ this is the blow-up at four points of $\F_4$. For $i=1$ we get the residual case. 

The cases (2), (3) and (4) each have one section $\Gamma$ in $\exc(\phi)$, and it is the only transversal curve. If $i=-d$, the result is immediate. Otherwise, since $-e=\Gamma^2 \leq -2$ (in $X$), we successively blow-down the floating curves into $\F_e$ where $\Gamma$ is the curve of negative self-intersection. This shows that these cases are also of type (I)' or (II)'.

Case (5) follows from Lemma \ref{caso(5)raro}. In case (6), we have a $\Gamma \sim C_{\infty}+(d+1)F$, and so $\Gamma^2=d+2$. Since $\Gamma$ belongs to $\exc(\phi)$, we must have an intermediate blow-up that makes the self-intersection of $\Gamma$ equal to $0$. Then $\Gamma$ and $C_{\infty}$ can be considered in $\F_d$ as a section and fiber, thus we reduce to a case that has a fiber and a section at least.
\end{proof}

The goal now is to \say{erase} the distinction between types (I)',(II)'' and marked surfaces of type (I), (II). For this, we need to produce a marking on a Wahl chain using deformations of surfaces $W^\prime_*$ of type (I)' or (II)'. The following general lemma is key to this process.  

Let $[e_1,\ldots,e_r]$ be an HJ continued fraction with $e_i\geq 2$, and let $[b_1,\ldots,b_s]$ be its dual. We know that $[e_1,\ldots,e_r,1,b_s,\ldots,b_1]=0$. Let us reproduce it over a fiber $F$ of a $\F_d \to \P^1$ for some arbitrary $d$. Thus, we have a genus $0$ fibration $X \to \P^1$ where all fibers are $\P^1$ except for the one representing $[e_1,\ldots,e_r,1,b_s,\ldots,b_1]=0$. Let $E_i$, $B_j$ be the corresponding exceptional curves, and let $C$ be the $(-1)$-curve. There exists a deformation $X_t \rightsquigarrow  X$ that is constant for the curves $B_j$ and eliminates all the curves $E_i$. There are no local-to-global obstructions to performing this.  

\begin{lemma}
There is a disjoint collection of $(-1)$-curves $\Gamma_j$ in $X_t$, each of which intersects transversally at one point the chain $B_1,\ldots,B_s$. The $\Gamma_j$s and the $B_i$s form a fiber of a genus $0$ fibration $X_t \to \P^1$. The Iitaka-Kodaira contractions over $X_t \rightsquigarrow X$ induced by $C$ and all the $(-1)$-curves after that over the fiber $F$ have the effect of contracting the $\Gamma_j$s and the $B_i$s into a $\P^1$ fiber. 
\label{keyContraction}
\end{lemma}

\begin{proof}
This is essentially the way one recovers the zero continued fraction of the P-resolution whose M-resolution is the minimal resolution. It is the simpler case of the geometric algorithm shown in \cite[Corollary 10.1]{PPSU18}.
\end{proof}

\begin{example}
Start with the chain $[2,\ldots,2]$, that is, $e_i=2$ for all $i$. Its dual is $[m]$ where $m-1$ is the number of $2$s. We construct the configuration $[2,\ldots,2,1,m]$ on a fiber of $X \to \P^1$, taking $B_1$ as the proper transform of the chosen fiber $F$ in some $\F_d \to \P^1$. 
Consider a deformation $X_t \rightsquigarrow  X$ such that $B_1 \subset X_t$ for all $t$, but the $(-2)$-curves do not lift to $X_t$ for $t \neq 0$. In the general fiber $X_t$, the lifted $(-1)$-curve intersects $B_1$ transversally and can be contracted in the family (Iitaka-Kodaira blow-down). After this contraction, a new lifted $(-1)$-curve appears in the resulting general fiber $X_t^\prime$, avoiding the point of the previous contraction, since no $(-2)$-curve lifts. We repeat this process until none of the $(-2)$-curves appear in the special fiber. In this way, we obtain a deformation $X_t \rightsquigarrow  X$ where $B_1$ meets $m$ disjoint $(-1)$-curves in the general fiber. 
\label{12...2}
\end{example}

The next theorem proves Theorem \ref{mainthm1}. 

\begin{theorem}
Let $W_*$ be a surface with one Wahl singularity, $K_{W_*}^2 \geq 1$, and $-K_{W_*}$ big. Then $W_*$ is a degeneration of a $\text{Bl}_u(W_{*m})$ for some $u \geq 0$, and some $W_{*m}$ or a residual surface.
\label{theTheorem}
\end{theorem}

\begin{proof}
By Lemma \ref{reductiontype}, it suffices to prove the statement for a $W_*$ of type (I)' or (II)' or a residual surface. Let us assume that $W_*$ does not have floating curves. This is possible because our goal is to construct a floating blow-up. Let $\phi \colon X \to W_*$ be the minimal resolution. In addition, there exists a contraction $\pi \colon X \to \F_d$ with no blow-ups over $C_{\infty}$, such that $\exc(\phi)$ contains $C_{\infty}$ as its only transversal curve, while the remaining exceptional curves lie in one or two fibers of the corresponding fibration $X \to \P^1$.

Consider a fiber $F^\prime$ that contains curves in $\exc(\phi)$. Since $F^\prime$ is a tree with simple normal crossings, any contraction of $(-1)$-curves contained in $F^\prime$ also yields a tree with the same property. The intersection $F^\prime \cap \exc(\phi)$ is a chain $E_1,\ldots,E_p$ where we choose that $E_p \cdot C_{\infty}=1$. Let $G^\prime=F^\prime \setminus (F^\prime \cap \exc(\phi))$.  

\textbf{Claim 1:} $G'$ is a disjoint union of chains of $\P^1$s. Take one connected component of curves in $G'$. Note that the only possible $(-1)$-curve in this component must intersect one $E_j$ since $W_*$ has no floating curves. Then, we verify that this component must be a chain, otherwise we would obtain a triple point after contracting $(-1)$-curves over $F'$.

\textbf{Claim 2:} At most one of the connected components in $G'$ that intersects a fixed curve $E_j$ is a chain with self-intersections $[2,\ldots,2,1]$. If there are two such, they do not contract completely before $E_j$ does. Thus, we generate a triple point in the blow-down of $(-1)$-curves over $F'$, a contradiction. 

The goal now is to construct a smooth deformation $X_t \rightsquigarrow  X$ over $\D$ that is constant on $E_1,\ldots,E_p$, and disjoint $(-1)$-curves in $X_t$ intersecting the $E_1,\ldots,E_p$. These curves contract together with $E_1,\ldots,E_{p-1}$ after performing all possible Iitaka-Kodaira blow-downs on curves in $F'$. We follow the steps:

\textbf{Step 1.} We proceed as in Example \ref{12...2} for each chain of type $[1,2,\ldots,2]$. We obtain a smooth deformation $X_t \rightsquigarrow  X$ that preserves $E_1,\ldots,E_p$ and every other curve in $F^\prime$, except the $(-2)$-curves in the chains $[1,2,\ldots,2]$. We make all Iitaka-Kodaira blow-downs associated to these chains, and possible $(-1)$-curves coming from the $E_j$. 

\textbf{Step 2.} We obtain a new chain $E_1,\ldots,E_q$ where $E_q \cdot C_{\infty}=1$, no $E_i$ is a $(-1)$-curve, and there are no chains of the form $[1,2,\ldots,2]$ in the special fiber. Since this configuration must contract into a fiber in $\F_d \to \P^1$, $(-1)$-curves must appear in the fiber. Thus, there must be chains in the new $G'$ that start with a $(-1)$-curve. One of them must intersect the top curve $E_1$ of the exceptional chain. Otherwise, we would create a triple point in the composition of blow-downs into $\F_d$.

\textbf{Step 3.} There is a chain in $G'$ with a $(-1)$-curve that intersects $E_1$. This entire chain must contract in the sequence of blow-downs with respect to $E_1,\ldots,E_q$, terminating at some $E_j$ with $j\leq q$. Thus, we are in the situation of Lemma \ref{keyContraction}. Here $B_1=E_1, \ldots, B_{j-1}=E_{j-1}$. We then obtain a smooth deformation $X'_t \rightsquigarrow  X'$ that preserves $E_1,\ldots,E_q$ and gives the disjoint $(-1)$-curves attached to the $E_1,\ldots,E_q$ in $X'_t$.

\textbf{Step 4.} We perform the Iitaka-Kodaira blow-downs corresponding to the top chain, and the $E_1,\ldots,E_{j-1}$. This brings us to a situation as in Step 2, and we repeat until no $(-1)$-curves remain in the special fiber.

In the process above, we assume that every deformation is constant on $C_{\infty}$ and on a potential second fiber, as in the (II)' case. We repeat the procedure for the second fiber. In the end, we obtain a smooth deformation $X_t \rightsquigarrow  X$ which preserves the Wahl chain, and it produces the marking of $(-1)$-curves. We then contract the Wahl chain in every fiber of the deformation, yielding a deformation $W_t$ into $W_*$, where $W_t$ is a $W_{*m}$ for the same Wahl singularity of $W_*$ and some marking. Since the process was carried out up to floating blow-ups, we obtain the desired deformation.
\end{proof}

\begin{corollary}
A W-del Pezzo surface of degree $\ell$ with a Wahl singularity $\frac{1}{n^2}(1,na-1)$ exists if and only if the Wahl chain of $\frac{n^2}{na-1}$ is del Pezzo of degree $\leq \ell$.
\label{characterization}
\end{corollary}

\begin{proof}
The Wahl chain $[4]$ exists as degeneration of $\P^2$, and by floating blowing-up it does for del Pezzos of any degree. When it is not $[4]$, we are left with the conclusion of Theorem \ref{theTheorem}, which is the content of the corollary. On the other hand, any del Pezzo Wahl chain produces a W-del Pezzo surface by Theorem \ref{theorem:W*m}, unless it is of type (I) and degree $1$. In this last case, by Lemma \ref{wirdo} we have some floating curve or we consider it as a type (II) degree $1$ del Pezzo, and so we obtain a W-del Pezzo in any case.
\end{proof}

\begin{corollary}
In degree $\ell$ greater than or equal to $5$, we have infinitely many Wahl singularities not allowed in degenerations of degree $\ell$ del Pezzo surfaces.  
\end{corollary}

\begin{proof}
By Corollary \ref{characterization}, the corresponding Wahl chain must be del Pezzo of degree $\geq 5$. But by Proposition \ref{5isnotfree} there are infinitely many that are not del Pezzo (see also Remark \ref{muchmore}). 
\end{proof}

\subsection{How to distinguish when $K^2=8$?} \label{ss5.1} For degree $8$ we have two types of nonsingular fibers in a W-surface: Hirzebruch surfaces $\F_k$ with $k$ odd or even. We can distinguish among them using only the marking of the Wahl chain. Let $W_{*m}$ be a marked del Pezzo of degree $8$, with Wahl singularity $\frac{1}{n^2}(1,na-1)$.   

\begin{theorem}
If $n$ is even, then $\F_k \rightsquigarrow W_{*m}$ with $k$ odd.
\label{K^28F1}
\end{theorem}

\begin{proof}
Since $H_1(\F_k,\Z)=0$, there is a generator for the class group of the singularity, which is image of a divisor $A$ in Cl$(W_{*m})$. By \cite[Lemma 3.1]{TU22}, we can choose $A$ so that in the minimal resolution $X \to W_{*m}$ it intersects the exceptional Wahl chain transversally at one point in an ending component, we consider the initial component of the chain. Then, we have $A \cdot K_{W_{*m}}=-\frac{a}{n}$. Therefore, $nA \in \Pic(\F_k)$ and $n A \cdot K_{\F_k}=-a$. Since $n$ is even, then $a$ must be odd and consequently $k$ must be odd. 
\end{proof}

\begin{theorem}
If $n$ is odd, then $\F_k \rightsquigarrow W_{*m}$ with $k$ even if and only if $n \Gamma \cdot K_{W_{*m}}$ is even for every marking curve $\Gamma$.
\label{K^28F0}
\end{theorem}

\begin{proof}
Let $k$ be even. Since $n \Gamma \in \Pic(\F_k)$, then $n \Gamma \cdot K_{W_{*m}}$ is even for curve $\Gamma$. Conversely, let $\{\Gamma_j\}_j$ be the marking curves and assume $n \Gamma_j \cdot K_{W_{*m}}$ is even for every $j$. We note that $\{\Gamma_j\}_j$ generates the class group Cl$(W_{*m})$, and this group contains $\Pic(\F_k)$. Hence, any generator of $\Pic(\F_k)$ is expressed in terms of $\Gamma_j$. It follows that the intersection of an element in $\Pic(\F_k)$ with the canonical class must be even, since $n$ is odd and $n \Gamma \cdot K_{W_{*m}}$ is even for the marking curves $\Gamma$. Therefore, $k$ must be even.
\end{proof}

%----------------------------------------------------------------------------
\section{Diophantine equations} \label{s7}

Let $0<a<n$ be coprime. Assume that the Wahl chain $\frac{n^2}{na-1}=[e_1,\ldots,e_r]$ is del Pezzo of type (II), degree $9-\lambda_1-\lambda_2$, and central mark $e_i=d$ (as in Definition \ref{DPWahlchain}). We have $d=8 +u_1-\lambda_1 + u_2-\lambda_2$ by Lemma \ref{restrCentralMark}. Let $$\frac{m_1}{q_1}=[e_{i-1},\ldots,e_1] \ \ \ \text{and} \ \ \ \frac{m_2}{q_2}=[e_{i+1},\ldots,e_r].$$

We recall the following well-known matrix factorization; see e.g. \cite[\S 2]{UV22}. 

\begin{proposition}
Let $0<q<m$ be coprime integers, and consider $\frac{m}{q}=[x_1,\ldots,x_s]$. Then $\left[\begin{array}{cc}
m & -q^{-1} \\
q & \frac{1-q q^{-1}}{m}
\end{array}\right] =
\left[\begin{array}{cc}
x_1 & -1 \\
1 & 0
\end{array}\right] \cdots\left[\begin{array}{cc}
x_s & -1 \\
1 & 0
\end{array}\right]$, 
 where $0<q^{-1}<m$ is the integer that satisfies $q q^{-1} \equiv 1($mod $m)$. \label{matrixhjcf}
\end{proposition}

Hence we have 
\[ 
\left[\begin{array}{cc}
n^2 & -n^2+na+1 \\
na-1 & -na+a^2+1
\end{array}\right] 
=\left[\begin{array}{cc}
m_1 & -q_1 \\
q_1^{-1} & \frac{1-q_1 q_1^{-1}}{m_1}
\end{array}\right]  \left[\begin{array}{cc}
d & -1 \\
1 & 0
\end{array}\right] \left[\begin{array}{cc}
m_2 & -q_2^{-1} \\
q_2 & \frac{1-q_2 q_2^{-1}}{m_2}
\end{array}\right]. \] This gives rise to four relations for each matrix entry $\left[\begin{array}{cc}
\text{1st} & \text{2nd} \\
\text{3rd} & \text{4th}
\end{array}\right]$.\\ The 1st gives $n^2=m_1 m_2 d -q_1m_2-q_2m_1$. Consider now the $2 \times 2$ system of equations on $q_1^{-1}$ and $q_2^{-1}$ given by adding and subtracting the 2nd and the 3rd relations. We obtain the expressions $m_1+m_2=n(m_1a-nq_1^{-1})=n(m_2(n-a)-nq_2^{-1})$. The 4th does not add anything new. 

Take the surface $\widehat{W}_{*m}$ defined by $W_{*m}$. In $\widehat{W}_{*m}$ we have partial resolutions of $\frac{1}{m_1}(1,m_1-q_1)$ and $\frac{1}{m_2}(1,m_2-q_2)$ whose exceptional curves are in the two special fibers of the minimal resolution of $\widehat{W}_{*m}$. Hence we have contractions $$\widehat{W}_{*m} \to W_{*m}^T \to P(n^2,m_1,m_2),$$ where $P(n^2,m_1,m_2)$ is a toric surface with Picard number $1$ and with $3$ singularities $\frac{1}{n^2}(1,na-1)$, $\frac{1}{m_1}(1,m_1-q_1)$ and $\frac{1}{m_2}(1,m_2-q_2)$. We actually have a characterization of marked surfaces through this matrix factorization. (Of course, we can write down a similar characterization for type (I), see Remark \ref{type(I)Diophantine}.)

\begin{theorem}
Let $P(n^2,m_1,m_2)$ be a toric surface of Picard number $1$, and singularities $\frac{1}{n^2}(1,na-1)$, $\frac{1}{m_1}(1,m_1-q_1)$, $\frac{1}{m_2}(1,m_2-q_2)$ forming the chain $$\frac{1}{m_1}(1,m_1-q_1) - (1) - \Big[ {n \choose a} \Big] - (1) - \frac{1}{m_2}(1,m_2-q_2^{-1}),$$ where $0<a<n$ and $0<q_i<m_i$ are coprime. Assume that
\begin{itemize}
    \item[(1)] $q_1m_2+q_2m_1+n^2=m_1 m_2 d$, for some $d\geq 2$. 
    \item[(2)] $m_1+m_2=n(m_1a-nq_1^{-1})=n(m_2(n-a)-nq_2^{-1})$. 
    \item[(3)] $\frac{1}{m_1}(1,q_1)$ and $\frac{1}{m_2}(1,q_2)$ admit zero continued fractions of weights $\lambda_1$ and $\lambda_2$, respectively, with $9-\lambda_1-\lambda_2>0$.
\end{itemize}
Then there is a partial resolution $W \to P(n^2,m_1,m_2)$ such that $W=\widehat{W}_{*m}$ for some marked surface $W_{*m}$ with Wahl singularity $\frac{1}{n^2}(1,na-1)$ and degree $9-\lambda_1-\lambda_2>0$.
\label{thm:diophequat}
\end{theorem}

\begin{proof}
We have $\left[\begin{array}{cc}
A & B \\
C & D
\end{array}\right] 
=\left[\begin{array}{cc}
m_1 & -q_1 \\
q_1^{-1} & \frac{1-q_1 q_1^{-1}}{m_1}
\end{array}\right]  \left[\begin{array}{cc}
d & -1 \\
1 & 0
\end{array}\right] \left[\begin{array}{cc}
m_2 & -q_2^{-1} \\
q_2 & \frac{1-q_2 q_2^{-1}}{m_2}
\end{array}\right]$.

By (1), we obtain $A=n^2$. Note that $m_1m_2(C+B)=(m_1 q_2^{-1}+m_2 q_1^{-1})n^2+m_1^2+m_2^2$, and $m_1m_2(C-B)=(-m_1 q_2^{-1}+m_2 q_1^{-1})n^2-m_1^2+m_2^2$. By evaluating in both equations the values of $q_1^{-1}$ and $q_2^{-1}$ from (2), we obtain $C-B=n^2-2$ and $C+B=2an-n^2$. Therefore, $C=na-1$ and $B=an+1-n^2$. Also, by evaluating $q_1^{-1}$ and $q_2^{-1}$ in $D$ we obtain $D=-na+a^2+1$.

Since we have the chain $\frac{1}{m_1}(1,m_1-q_1) - (1) - \big[ {n \choose a} \big] - (1) - \frac{1}{m_2}(1,m_2-q_2^{-1})$, the minimal resolution of $P(n^2,m_1,m_2)$ contains two genus zero fibers in the toric boundary, formed by the minimal resolutions of $\frac{1}{m_i}(1,m_i-q_i)$, its dual, and a $(-1)$-curve. We contract curves in these two fibers until we obtain the Hirzebruch surface $\F_d$. Since the $\frac{1}{m_i}(1,q_i)$ admit zero continued fractions of weights $\lambda_i$, then their duals $\frac{1}{m_i}(1,m_i-q_i)$ admit the corresponding (and unique) M-resolutions. Consider the partial resolution $W' \to P(n^2,m_1,m_2)$ defined by both of them. Let $W''$ be the minimal resolution of $W'$. There are no local-to-global obstructions to deform $W''$, and so we consider a $\Q$-Gorenstein smoothing $W_t \rightsquigarrow W''$ that keeps the Wahl chain of $n,a$. As in \cite[Corollary 10.1]{PPSU18}, we run MMP on $W_t \rightsquigarrow W''$ to define a marking over the Wahl chain in $W_t$. Then we contract the Wahl chain on the original deformation $W_t \rightsquigarrow W''$, and we obtain on the general fiber a marked surface $W_{*m}$. On the other hand, the surface $W_{*m}$ defines a $\widehat{W}_{*m}$, and this is our claimed partial resolution of $P(n^2,m_1,m_2)$. 
\end{proof}

\begin{remark}
When gcd$(m_1,m_2)=\mu>1$, the surface $P(n^2,m_1,m_2)$ is a fake weighted projective plane. Then we have a cyclic covering $\mathbb{P}(n^2/\mu,m_1/\mu,m_2/\mu) \to P(n^2,m_1,m_2)$ of degree $\mu$ (see the proof of Theorem 4.1 in \cite{HP10}). 
\label{fakeprojplanes}
\end{remark}

\begin{remark}
For type (I) we have the chain of Wahl singularities $$(0) - \Big[ {n \choose a} \Big] - (1) - \frac{1}{m}(1,m-q^{-1}),$$ where $m=na-1$, $q=d(na-1)-n^2$, and $q^{-1}=an-a^2-1$. The corresponding toric variety is $\P(1,n^2,na-1)$. Note that $d=u+l-3$ for some $u \geq 0$.
\label{type(I)Diophantine}
\end{remark}

\begin{remark}
Given a marked surface $W_{*m}$ of degree $\ell$ with associated surface $P(n^2,m_1,m_2)$, consider a general fiber $F$ of $X \to \P^1$, where $X$ is the minimal resolution of $W_{*m}$. Then, the Hodge index theorem applied to $F$ and the ample divisor $-K_{W_{*m}}$ implies that $(n^2+m_1+m_2)^2 \geq \ell n^2m_1m_2$. This is a nontrivial inequality when $\ell \geq 5$. In the case of type (I) we obtain $(n+a)^2\geq \ell (an-1)$.  
\label{hodgeIndex}
\end{remark}

\begin{example}
Let us take the Wahl chain $[3,2,8,2,2,2,4,2]$, where $n=27$, $a=11$. It has the marking $[1,1,\underline{8},2,2,1,4,1]$ of degree $6$. We have a $\widehat{W}_{*m}$ with chain of Wahl singularities $(2)-(1)-\Big[{3 \choose 1}\Big]-(1)-\Big[{27 \choose 11}\Big]-(1)-\Big[{7 \choose 3}\Big]-(1)$. We have $d=8$, $\frac{m_1}{q_1}=\frac{5}{3}=[2,3]$, $\frac{m_2}{q_2}=\frac{22}{17}=[2,2,2,4,2]$, $q_1^{-1}=2$, and $q_2^{-1}=13$. The associated $P(27^2,5,22)$ is $\P(27^2,5,22)$.    
\end{example}

\begin{remark}
The case of degree $9$ is very special. As $\lambda_1=\lambda_2=0$, we have $m_i=n_i^2$ and $q_i=n_i(n_i-a_i)+1$ for some $0<a_i<n_i$ coprime. Then (2) in Theorem \ref{thm:diophequat} gives $a_1= \frac{n_1^2na-n^2-n_1^2-n_2^2}{n^2n_1}$ and $a_2=\frac{n_2^2n(n-a)-n^2-n_1^2-n_2^2}{n^2n_2}$, and by evaluating them in (1) we obtain $$(n^2+n_1^2+n_2^2)^2-(d-1)n^2n_1^2n_2^2=0.$$ But $d=10$ in degree $9$, and so $(n^2+n_1^2+n_2^2+3n_1n_2n)(n^2+n_1^2+n_2^2-3n_1n_2n)=0.$ Therefore $(n_1,n_2,n)$ must be a Markov triple. In more generality, if the singularities $\frac{1}{m_i}(1,q_i)$ are T-singularities, then one has the equation $$(n^2+d_1 n_1^2+ d_2 n_2^2)^2-(11-d_1-d_2)d_1 d_2 n^2n_1^2n_2^2=0,$$ where $m_i=d_i n_i^2$. In this case, $d=12-d_1-d_2$, and the degree is $\ell=11-d_1-d_2$. This way, one recovers Table 1 and Table 2 in \cite{HP10} for toric surfaces with Picard number one and at least one Wahl singularity. This happens exactly when $\ell d_1 d_2$ is a square, and we obtain the Markov type equation in each case.
\label{markovequation}
\end{remark}

%--- Maybe here we can compute the result of a top smoothing and slide, as a mutation of $\frac{1}{n^2}(1,na-1)$, $\frac{1}{m1}(1,m_1-q_1)$ and $\frac{1}{m_2}(1,m_2-q_2)$.

\begin{example}
Let us consider a $W_{*}$ of type (II) and degree $8$, with Wahl singularity $\frac{1}{n_1^2}(1,n_1 a_1-1)$. We have $9-\lambda_1-\lambda_2=8$, and so say $\lambda_1=0$ and $\lambda_2=1$. Its toric surface $\widehat{W}_{*m}$ has the chain of Wahl singularities: $$\left[{n_0 \choose a_0}\right]-(1)-\left[{n_{1} \choose a_{1}}\right]-(1)- \left[{n_{2} \choose a_{2}}\right]-(1)-\left[{n_{3} \choose a_{3}}\right],$$ where $n_0>1$. We have $m_1=n_0^2$, $m_1-q_1=n_0 a_0 -1$, $\Delta:=m_2$, $\Omega^{-1}:=m_2-q_2$. In this way, $\left[{n_{2} \choose a_{2}}\right]-(1)-\left[{n_{3} \choose a_{3}}\right]$ contracts to $\frac{1}{\Delta}(1,\Omega)$. Let $-d$ be the self-intersection of the central mark. Define $\delta=n_0 a_1 - n_1 a_0$ and $\delta'=n_2a_3 - n_3 a_2$. Then $\Delta= \delta n_0 n_1 - n_0^2 - n_1^2 = n_2^2 + n_3^2 - \delta' n_2 n_3$ and $\Omega= -\delta' n_3 a_2 +n_3 a_3 + n_2 a_2 -1=\delta n_0a_1-n_1a_1-n_0a_0-1$, and so we have the expressions $$\delta n_0 n_1 + \delta' n_2 n_3= n_0^2+n_1^2+n_2^2+n_3^2$$ and $\delta n_0 a_1 + \delta' a_2 n_3= n_0a_0+n_1a_1+n_2a_2+n_3a_3$.
\label{degree8}
\end{example}

\begin{remark}
All possible marked Wahl chains of degree $8$ are:
\begin{itemize}
\item[(I)] Wahl chains $[e_1,\ldots,e_r]$ such that $[e_2,\ldots,e_r]$ admits a zero continued fraction of weight $1$, and
\item[(II)] Wahl chains of the form $\left[\Wa^{\vee},d,\frac{\Delta}{\Omega} \right],$ where $\Wa^{\vee}$ is a dual Wahl chain, $d \geq 2$, and $\frac{\Delta}{\Delta-\Omega}$ admits an extremal P-resolution, or is a T-singularity (as in Remark \ref{markovequation}, and so in this case the indices $(n_0,n_1,n_2)$ satisfy $n_0^2+n_1^2+2n_2^2=4n_0 n_1 n_2$).
\end{itemize}
\label{beatdegree8}
\end{remark}

%----------------------------------------------------------------------------
\section{Exceptional vector bundles on del Pezzo surfaces} \label{s8}

The purpose of this section is to provide another proof of \cite[Theorem A]{PR24} in the context of marked surfaces $W_{*m}$ and their deformations. For example, we will extend that result to a strong full exceptional collection through the canonical marking.

\begin{definition}
An exceptional vector bundle (e.v.b.) $E$ on a projective surface $Y$ is a locally free sheaf such that $\Hom(E,E)=\C$, and $\Ext^1(E,E)=\Ext^2(E,E)=0$. \label{evb}
\end{definition}

An exceptional vector bundle $E$ on $Y$ induces the semi-orthogonal decomposition (s.o.d.) $D^b(Y)=\langle E^{\perp},E \rangle$, where $E$ is the admissible triangulated subcategory generated by $E$ and $E^{\perp}=\{ t \in D^b(Y) \colon \Hom(e,t)=0, \ \forall \ e\in E \}$; see \cite[\S 1.4]{Huy06}. Hacking \cite{H13} gives a way to construct exceptional vector bundles on nonsingular fibers of W-surfaces when $p_g=q=0$.

\begin{theorem}[Theorem 1.1 \cite{H13}]
Let $(W \subset \W) \to (0 \in \D)$ be a W-surface such that 
\begin{itemize}
    \item[(1)] $W$ has only one Wahl singularity $\frac{1}{n^2}(1,na-1)$, $p_g(W)=q(W)=0$, and
    \item[(2)] the induced exact sequence $0 \to H_2(W_t) \to H_2(W) \to \Z/n \to 0 $ is exact, where $t \neq 0$. (We have $\Pic(W_t)=H_2(W_t)$ and $\Cl(W)=H_2(W)$.)
\end{itemize}

Then, possibly after a base change $\D \to \D$, there exists a reflexive sheaf $\cE$ (i.e. $\cE \to \cE^{\vee \vee}$ is an isomorphism) on $\W$ such that 
\begin{itemize}
    \item[(a)] $E:= \cE|_{W_t}$ is an e.v.b. of rank $n$, and
    \item[(b)] $E_0:= \cE|_{W}$ is a torsion-free sheaf on $W$ such that $E_0^{\vee \vee}$ (reflexive hull) is isomorphic to the direct sum of $n$ copies of a reflexive rank $1$ sheaf $A$, and the quotient $E_0^{\vee \vee}/E_0$ is a torsion sheaf supported at $\frac{1}{n^2}(1,na-1)$. 
\end{itemize}

If $\mathcal{H}$ is a line bundle on $\W$ which is ample on the fibers, then $E$ is slope stable with respect to $\mathcal{H}|_{W_t}$. Moreover we have $$c_1(E)=n c_1(A) \ \ \ \ \ \ \ \ c_2(E)=\frac{n-1}{2n}\big(c_1^2(E) +n+1 \big) \ \ \ \ \ \ \ \ c_1(E) \cdot K_{W_t} \equiv \pm a (\text{mod} \ n). $$ 
    \label{h.e.v.b.}
\end{theorem}

A \textit{Hacking vector bundle} on a nonsingular projective surface $Y$ is an e.v.b. isomorphic to some e.v.b. $E$ as in Theorem \ref{h.e.v.b.}. Assumption (2) in Theorem \ref{h.e.v.b.} is satisfied, for example, when $H_1(W_t)=0$.

\begin{definition}
An \textit{exceptional collection of vector bundles} (e.c.) on a projective surface $Y$ is a collection of e.v.b. $E_r,E_{r-1},\ldots, E_0$ such that $\Ext^k(E_i,E_j)=0$ for all $i<j$ and all $k\geq 0$. Its \textit{length} is $r+1$. It is said to be \textit{strong} if moreover $\Ext^k(E_i,E_j)=0$ for all $i>j$, and all $k>0$. 
    \label{e.c.}
\end{definition}

As explained in \cite[\S 1.4]{Huy06}, an exceptional collection $E_r,E_{r-1},\ldots, E_0$ on $Y$ defines a s.o.d. $D^b(Y)=\langle \cA^{\perp}, E_r,E_{r-1},\ldots, E_0, \rangle$ where $\cA$ is generated by $E_r,E_{r-1}$, $\ldots, E_0$. The exceptional collection is \textit{full} if $\cA^{\perp}=\emptyset$. Any exceptional collection on a del Pezzo surface can be completed as a full exceptional collection by \cite{KO94}. Hence, on a del Pezzo surface of degree $\ell$, an exceptional collection of maximal length is full. This length must be equal to $12-\ell$.

In fact, given a full e.c. $F_r,\ldots,F_0$ on a del Pezzo surface, there is a toric variety $W$ with only Wahl singularities of indices $n_i=\rk(F_i)$ such that all $F_i$ are Hacking vector bundles for a $\Q$-Gorenstein smoothing of $W$. See \cite[Section 2.5.1]{H12}. This exceptional collection can be explicitly constructed on many W-surfaces with a chain of Wahl singularities. A general recipe is the following. %\red{In particular, any e.v.b. on a del Pezzo surface arises as a Hacking vector bundle (see Theorem \ref{HecfordelPezzo}).} 

\begin{theorem}[Theorems 5.5 and 5.8 \cite{TU22}] Let $W_t \rightsquigarrow W$ be a W-surface such that 

\begin{itemize}
    \item[(1)] $p_g(W)=q(W)=0$,
    \item[(2)] the surface $W$ has exactly the Wahl singularities $P_0,\ldots,P_r$ (we also allow $P_i$ to be smooth points), and a chain of nonsingular rational curves $\Gamma_1,\ldots,\Gamma_r$ that are toric boundary divisors $\Gamma_i$, $\Gamma_{i+1}$ at $P_{i+1}$ (a chain of Wahl singularities Definition \ref{def:chainWahlsing}), and
    \item[(3)] there exists a Weil divisor $A \subset W$, which is Cartier outside of $P_0$ and generates the local class group $\Cl(P_0 \in W)$. 
\end{itemize}
Then, after possibly shrinking $\D$, there exists an e.c. $E_r,\ldots,E_0$ of Hacking vector bundles on $W_t$, $t \neq 0$, with
$$\rk(E_i)=n_i,\quad
c_1(E_i)=-n_i(A+\Gamma_1+\ldots+\Gamma_i) \in H^2(W_t,\Z),$$ where $P_i=\frac{1}{n_i^2}(1,n_i a_i-1)$. (For $P_i$ smooth we take $n_i=a_i=1$.) By Riemann--Roch, we also have
$c_2(E_i)=\frac{n_i-1}{2n_i}(c_1(E_i)^2 + n_i +1)$. We have that $c_1(E_i) \cdot K_{W_t} \equiv a_i (\text{mod} \ n_i)$.
\label{H.e.c.theorem}    
\end{theorem}

\begin{definition}
A \textit{Hacking exceptional collection} (H.e.c.) on a nonsingular projective surface $Y$ is the existence of a W-surface $W_t \rightsquigarrow W$ where $Y=W_t$ for some $t \neq 0$, and an e.c. as in Theorem \ref{H.e.c.theorem} on $Y$.
\label{H.e.c.}
\end{definition}

In Theorem \ref{H.e.c.}, the hypothesis (3) is satisfied automatically when $W_t$ is rational since $H^1(W_t,\Z)=0$. If $n_0=1$ ($P_0$ is a nonsingular point), then we can tensor the H.e.c. by $A$ being linearly equivalent to $\O_{W_t}$. 

\begin{definition}
Let $Y$ be a nonsingular del Pezzo surface. We define the slope of a vector bundle $E$ as $\mu(E):=\frac{\deg(E)}{\rk(E)}$ where $\deg(E):= - c_1(E) \cdot K_Y$.
\label{slope}
\end{definition}

\begin{theorem}[Corollary 2.5 \cite{Goro89}]
An exceptional vector bundle $E$ on a del Pezzo surface $Y$ is uniquely determined up to isomorphism by its slope $\frac{c_1(E)}{\rk(E)} \in H^2(Y,\Q)$. 
\label{Goro}
\end{theorem} 

\begin{theorem} [Theorem 2.5.3 and Corollary 2.5.4 \cite{H12}] Let $Y$ be a del nonsingular del Pezzo surface and $F_r,\ldots,F_0$ be a full e.c. on $Y$ ($r=11-K_Y^2$). Then, there is a toric variety $W$ with only Wahl singularities, and a W-surface $W_t \rightsquigarrow W$ such that $W_t$ is deformation equivalent to $Y$, and 
a H.e.c. $E_r,\ldots,E_0$ on $W_t$ such that $\frac{c_1(F_i)}{\rk(F_i)}=\frac{c_1(E_i)}{\rk(E_i)}$ in $H^2(Y,\Q)$ for all $i$. If we order the chain of Wahl singularities as $$\left[{n_0 \choose a_0}\right]-(c_1)-\left[{n_{1} \choose a_{1}}\right]-(c_2)- \ldots-(c_r)-\left[{n_{r} \choose a_{r}}\right]$$ using the toric boundary $\Gamma_1,\ldots,\Gamma_r$, then $\rk(F_i)=\rk(E_i)=n_i$ and $\deg(F_i)\equiv \deg(E_i) \equiv -a_i (\text{mod} \ n_i)$ for all $i$.  
%$F_r,\ldots,F_0$ is deformation equivalent to   
\label{HecfordelPezzo}
\end{theorem}

\begin{proof}
For the construction of $W$ see \cite[Theorem 11.3]{Per18} and \cite[Theorem 2.5.3]{H12}. From this last reference, one forms the vectors $u_i:=\frac{c_1(F_i)}{\rk(F_i)} - \frac{c_1(F_{i-1})}{\rk(F_{i-1})} \in H^2(Y,\Q)$ for $i=1,\ldots,r$, together with $u_0:=-K_Y-\sum_{j=1}^r u_j$. Let $M$ be the kernel of the homomorphism $\Z^{r+1} \to H^2(Y,\Q)$ sending $e_i$ to $u_i$. Let $\Z^{r+1} \to L$ defined by the dual of $M \subset \Z^{r+1}$, sending $e_i$ to $v_i$. Then $L$ is a free abelian group of rank $2$, and the vectors $v_i \in L$ are primitive and generate the rays of a complete fan in $L \otimes \R$ in cyclic order. This is the fan of $W$. We have $K_W^2=K_Y^2$. It is proved (\cite[Theorem 11.3]{Per18}) that the singularities of $W$ correspond to Wahl singularities $\frac{1}{n_i^2}(1,n_ia_i-1)$ where $n_i=\rk(F_i)$ (including the case of nonsingular points). Let $\Gamma_0,\ldots, \Gamma_r$ be the curves in the toric boundary of $W$. We consider the chain of Wahl singularities $$\left[{n_0 \choose a_0}\right]-(c_1)-\left[{n_{1} \choose a_{1}}\right]-(c_2)- \ldots-(c_r)-\left[{n_{r} \choose a_{r}}\right].$$
Since $-K_W$ is big, by \cite[Proposition 3.1]{HP10} there are no local-to-global obstructions to deform $W$. We may therefore consider a W-surface $W_t \rightsquigarrow W$. It satisfies the hypothesis of Theorem \ref{H.e.c.theorem} because $H^1(W_t,\Z)=0$. Hence, $W_t$ admits a H.e.c. $E_r,\ldots E_0$ and all the assertions follow from Theorem \ref{H.e.c.theorem} and Theorem \ref{Goro}.
\end{proof}

Since any e.c. on a del Pezzo surface can be completed to a full e.c. \cite{KO94}, then every e.v.b. on a del Pezzo surface arises as a Hacking vector bundle. We note that by \cite[2.2.3 Corollary]{Goro89} $c_1^2(E)$ and $\rk(E)$ of Theorem \ref{Goro} are coprime. Theorem \ref{HecfordelPezzo} also shows that $\deg(E)$ and $\rk(E)$ are coprime.  

\begin{corollary}
Let $W$ be a weak del Pezzo toric surface with only Wahl singularities, and let $W_t \rightsquigarrow W$ be a W-surface such that $W_t$ is a del Pezzo surface. Then $W_t$ admits a full and strong H.e.c. $E_r,\ldots,E_0$ such that for all $j>i$, we have $\hom(E_j,E_i)=-\rk(E_j) \rk(E_i) (\Gamma_{i+1}+\ldots+\Gamma_j)\cdot K_W.$
\label{HecWeakdelPezzo}
\end{corollary}

\begin{proof}
Let $\Gamma_0,\ldots,\Gamma_r$ be the toric boundary with $r=11-K_W^2$. The hypotheses of Theorem \ref{H.e.c.theorem} hold, and so we obtain the corresponding H.e.c. $E_r,\ldots,E_0$. Since this e.c is of maximum length, by \cite{KO94} it must be full. Following \cite[Corollary 2.11]{KO94}, if $E,F$ form an e.c. on a del Pezzo surface, then $\Ext^2(E,F)=0$, and $\Ext^i(E,F) \neq 0$ for at most one $i=0,1$. By the Riemann-Roch theorem, we have $$\sum_{i=0}^2 (-1)^i \ext^i(E,F)=\chi(E,F)=\chi(F,E)+ \rk(F) c_1(E) \cdot K_Y - \rk(E) c_1(F) \cdot K_Y,$$ thus, $\hom(E,F)-\ext^1(E,F)=\rk(F) \rk(E) (\mu(F)-\mu(E))$. It follows that $\mu(E) \leq \mu(F)$ is equivalent to $\hom(E,F)=\rk(F) \rk(E) (\mu(F)-\mu(E))$ and $\Ext^1(E,F)=0$. By construction, for $i<j$ we have $$ \frac{c_1(E_i)}{n_i}-\frac{c_1(E_j)}{n_j}=\Gamma_{i+1}+\ldots+\Gamma_j,$$ so $\mu(E_j) \leq \mu(E_i)$. This implies that the e.c. $E_r,\ldots,E_0$ is also strong.
\end{proof}

\begin{corollary}
Let $W_{*m}$ be a marked surface that is not of type (I) and degree $1$ (see Lemma \ref{wirdo}). Consider a general W-surface $W_t \rightsquigarrow W_{*m}$ so that $W_t$ is a del Pezzo surface. Then $W_t$ admits a H.e.c. that is full and strong.   
\label{HecW*m}
\end{corollary}

\begin{proof}
By Theorem \ref{theorem:W*m} the surface $W_{*m}$ is del Pezzo, and by Theorem \ref{theoremtoric} the toric surface $\widehat{W}_{*m}$ is weak del Pezzo. Then, it follows from Corollary \ref{HecWeakdelPezzo}.
\end{proof}

\begin{theorem}
Let $0<a<n$ be coprime integers. Then there is a full and strong H.e.c. $E_7,\ldots,E_1,\O_Y$ on del Pezzo surfaces $Y$ of degree $4$, where $\rk(E_1)=n$ and $c_1(E_1) \cdot -K_Y \equiv \pm a (\text{mod} \ n)$. 
\end{theorem}

\begin{proof}
If $n=2$, then $a=1$, and we can easily construct through a suitable blow up of $\F_4$ and Theorem \ref{H.e.c.theorem}. The same for $n=3$ and $a=1$ starting with $\F_5$. For any other $0<a<n$, we have two canonical markings by Lemma \ref{canmarking}. Consider the surfaces $\widehat{W}_{*m}$ associated to each marking, and construct the corresponding two H.e.c. $E_7,\ldots,E_1,\O_Y$ on del Pezzo surfaces $Y$ of degree $4$. Since both are of type (I), we start the chain of Wahl singularities with a nonsingular point, thus the collection begins with $\O_Y$. We have $\rk(E_1)=n$ and $c_1(E_1) \cdot -K_Y \equiv \pm a (\text{mod} \ n)$, since both canonical markings read the Wahl chain in its two distinct directions. These collections are full and strong by Corollary \ref{HecW*m}.  
\end{proof}

Next, we present two geometric proofs of \cite[Theorem A]{PR24}. The first of them proceeds by constructing a new family of surfaces derived from the canonical markings of Lemma~\ref{canmarking}.

\begin{definition}
A Wahl chain is of \textit{fiber type} if it admits a zero continued fraction of weight $\lambda$ with $\lambda \leq 8$. Its degree is $\ell=8-\lambda$. 
\label{def:fibertype} 
\end{definition}

For any $\F_d \to \P^1$, we choose a fiber and construct the chain of curves a given Wahl chain of fiber type. Its contraction is a surface $W_{*mf}$ with the corresponding Wahl singularity. This surface has $K_{W_{*mf}}^2=8-\lambda$ and $-K_{W_{*mf}}$ big. For each $d$, we can consider W-surfaces $W_t \rightsquigarrow W_{*mf}$ where $W_t$ is a del Pezzo surface. 

\begin{lemma}
A Wahl chain of fiber type is of degree at most $5$, with $[4]$ being the only chain of degree $5$. Any other Wahl chain is of fiber type of degree $4$.
\label{lem:fibertype}
\end{lemma}

\begin{proof}
Over the fiber, we first blow-up maximally to obtain $[1,2,\ldots,2,1]$ of length $u+1$. We then blow-up at $t$ nodes of the subsequent configurations, followed by $\lambda+1$ additional blow-ups on the interior of the marked components. If $[e_1,\ldots,e_r]$ is the Wahl chain, then $\sum_{i=1}^r e_i=2u+3t+\lambda+1$. Since $\sum_{i=1}^r e_i=3r+1$ and $r=u+1+t$, it follows that $\lambda=u+3$ and $\ell=5-u$.
If $\ell=5$, then $u=0$ and we have the Wahl chain $[4]$ only.
For $\ell=4$ use the canonical marking without the ending mark.  
\end{proof}

\begin{theorem}
Let $\partial, n$ be coprime integers with $n>0$. Then there is an e.c. $E,\O_Y$ on del Pezzo surfaces $Y$ of degree $4$ with $\rk(E)=n$ and $\deg(E)=\partial$.  
\label{Polishchuk}
\end{theorem}

\begin{proof}[Proof 1]
Consider $W_{*mf}$ for a given Wahl chain via the canonical fiber type and assume that we started with $\F_d$, $d\geq 0$. Let $\Gamma_1 \subset W_{*mf}$ denote the image of the section $C_\infty$. We apply Theorem \ref{H.e.c.theorem} to obtain an e.c. $E,\O_Y$ where $Y=W_t$ is a del Pezzo surface from a W-surface $W_t \rightsquigarrow W_{*mf}$. Therefore, $c_1(E) \cdot (-K_Y)= a+n(2-d)$ or $n-a+n(2-d)$ (using the flipped Wahl chain). Applying the operations $(E,\O_Y) \mapsto (\O_Y,E^{\vee}) \mapsto (E^{\vee}, \O_Y(-K_Y)) \mapsto (E^{\vee} \otimes \O_Y(K_Y), \O_Y)$ we obtain the e.v.b. $E^{\vee} \otimes \O_Y(K_Y)$ of rank $n$ and degree $-a-n(2-d)-4n$ or $-(n-a)-n(2-d)-4n$. This way, we exhaust all possible degrees.    
\end{proof}

\begin{proof}[Proof 2]
Let $n\geq 4$ (use previous proof for $n=2,3$). We consider the canonical markings from Lemma \ref{canmarking}, and the corresponding surfaces $W_{*m}^1$ and $W_{*m}^2$. For each $W_{*m}^i$, we identify particular $\P^1$s to apply Theorem \ref{H.e.c.theorem}. Let us consider $u$ as in Lemma \ref{restrCentralMark}. Since our marking is of type (I) and degree $\ell=4$, we note that $u+1=d$, where $-d$ is the self-intersection of the central mark. Recall that $u+1$ corresponds to the number of curves in the first $u$ consecutive blow-ups that will belong to the Wahl chain. Due to the marking, we perform an extra blow-up at the ending curve not intersecting $C_{\infty}$. We have the chain $[1,2,\ldots,2,1,d]$ where the number of $2$s is $d$. It can be numerically contracted to $[1,2,1,1]$, which in turn contracts to $[0,0]$, corresponding geometrically to $\P^1 \times \P^1$.
Then, for any $t\geq 0$, we can consider a smooth rational curve $\Gamma_1$ of bidegree $(t,1)$ not passing through the intersection point of the fiber and the section (two rulings). Additionally, let  $\Gamma'_1$ be another curve of bidegree $(t,1)$ passing through that intersection point. We have that $K_{\P^1 \times \P^1} \cdot (t,1)=-2t-2$.

Let $X_i \to W_{*m}^i$ be the minimal resolutions. Then we have $\Gamma_1 \cdot K_{X_i}=-2t-2$ and $\Gamma'_1 \cdot K_{X_i}=-2t-1$. We now apply Theorem \ref{H.e.c.theorem} for the chain of Wahl singularities given by $\Gamma_1 \subset W_{*m}^1$ and by $\Gamma'_1 \subset W_{*m}^1$, and analogously for $W_{*m}^2$. Combining these four cases, we construct e.c.s $E,\O_Y$ on a del Pezzo surface $Y$ of degree $4$ with $\deg(E)=n-a+ns$ and $a+ns$ for all $s\leq -1$. Again, applying the operations $(E,\O_Y) \mapsto (E^{\vee} \otimes \O_Y(K_Y), \O_Y)$ we obtain e.v.b.s $E^{\vee} \otimes \O_Y(K_Y)$ of rank $n$ and degrees $n-a+n(s-5)$ and $a+n(s-5)$ for all $s\geq 1$. Therefore, we obtain any $\partial$ coprime to $n$ in this way.  
\end{proof}

Following the results in the previous sections, we now prove that there are strong restrictions on the slope of an exceptional vector bundle on a del Pezzo surface of degree $\ell \geq 5$. This is stronger than \cite{PR24} as we do not impose the e.v.b. $E$ to be part of an e.c. $E,\O_Y$.

\begin{theorem}
There exist infinitely many pairs $\partial,n$ of coprime integers with $0<\partial<n$ such that an e.v.b. $E$ with $\rk(E)=n$ and $\deg(E)\equiv \pm\partial \pmod{n}$ is not realizable on any del Pezzo of degree $\geq 5$.
\label{strongerPolishcuk}
\end{theorem}

\begin{proof}
Assume that $E$ is an e.v.b. on a del Pezzo surface of degree $\ell \geq 5$ with $\rk(E)=n$ and $\deg(E)\equiv\pm\partial \pmod{n}$. By Theorem \ref{HecfordelPezzo} there is a toric surface $W$ that contains the corresponding Wahl singularity $\frac{1}{n^2}(1,n\partial-1)$. Thus, there exists a surface $W_*$ containing that Wahl singularity only, $K_{W_*}^2=K_W^2 \geq 5$ and $-K_{W_*}$ big. Then, $W_*$ is a degeneration of a $\text{Bl}_u(W_{*m})$ for some $u\geq 0$ and some $W_{*m}$ by Theorem \ref{theTheorem} (the residual surface is ignored at this point as the Wahl chain is $[4]$). By assumption, there is a marking on the Wahl chain that makes it a del Pezzo of degree $\geq 5$. However, by Proposition \ref{5isnotfree} (and Remark \ref{muchmore}) there are infinitely many Wahl chains that are not del Pezzo of degree greater than or equal to $5$.
\end{proof}

\begin{remark}
From a marked surface $W_{*m}$ with Wahl singularity $\frac{1}{n^2}(na-1)$ of type (I) and degree $\ell$, one can always construct $E,\O_Y$ e.c. with rank$(E)=n$ and $\deg(E)$ equal to $a+n(1-\ell)$ and $-a-n$ using a section disjoint to $C_{\infty}$ or a fiber. By Remark \ref{hodgeIndex}, we must have $(a+n)^2 \geq \ell(na-1)$. (In \cite[(5.1)]{PR24} a similar inequality is deduced using Riemann-Roch and Hodge index theorems.) To illustrate this, take $t\geq 1$ and the Wahl chain $[t+2,5,2,\ldots,2]$ with $n=2t+3$ and $a=2$. It has a type (I) marking of degree $5$: $[\underline{t+2},1,2,\ldots,2,1]$. Then the degrees $a+n(1-\ell)=-10-8t$ and $-a-n=-5-2t$ do not satisfy the second inequality \cite[(5.3)]{PR24}, since $$ \deg(E)^2+5 n \deg(E)+5n^2+1=4(t-1)(t+1) \geq 0.$$ 

In degree $\ell=5$, we also have that the following type (I) infinite families do not satisfy the second inequality \cite[(5.3)]{PR24}:
\begin{itemize}
    \item $[\underbrace{t+2,2,\ldots,2}_{t},t+4,\underbrace{2,\ldots,2}_{t}]$ with marking $[\underline{t+2},1,2,\ldots,2,t+1,1,2,\ldots,2]$. In this case $n=(t+1)^2+1$, $a=t+1$, and degrees are $-4t^2-7t-7$ and $-t^2-3t-3$.
    \item $[\underbrace{t+2,2,\ldots,2}_{t-1},t+2,\underbrace{2,\ldots,2,t+1}_{t-2},\underbrace{2,\ldots,2}_{t}]$ with marking 
    $[\underline{t+2},2,\ldots,2,1,t-1,2,\ldots,2,t+1,1,2,\ldots,2]$. We have $n=t^3-t^2+t+1$ and $a=t^2-2t+2$, and degrees $-4t^3+5t^2-6t-2$ and $-t^3+t-3$.
\end{itemize}

There are other ways to obtain $E,\O_Y$. For the family above $[t+2,5,2,\ldots,2]$, in $\F_{t+1}$ we can consider a section $\Gamma$ with $\Gamma \cdot C_{\infty}=1$, and blow up over $\Gamma \cap F$ for some fiber $F$. Since we have a marked $(-1)$-curve in the last $2$, we obtain a $\Gamma$ in $X$ with $\Gamma^2=3$. Then the construction of $E,\O_Y$ from that $\Gamma$ gives $deg(E)=-8t-14$ and $\rk(E)=2t+3$. The second inequality in \cite[(5.3)]{PR24} does not hold. 
\label{examplesPolishchuk}
\end{remark}

\begin{remark}
In \cite[Proposition 5.2]{PR24}, Polishchuk and Rains prove that for any $\ell \geq 5$ and any pair $(\partial, n)$ of relatively prime integers with $n> 0$ satisfying $$-\ell \leq {\partial}^2 + \ell n \partial + \ell n^2 \leq -1,$$ there exists an e.c. $E,\O_Y$ on a del Pezzo surface $Y$ of degree $\ell$ with $\partial=\deg(E)$ and rank$(E)=n$. The first inequality is necessary for this e.c. to exist and comes from the Riemann-Roch and Hodge index theorems. The second is arbitrary. To handle these types of integers, they consider integers in $\Q(\sqrt{\ell(\ell-4)})$ and their norms, which are precisely ${\partial}^2 + \ell n \partial + \ell n^2$. We now show how to obtain the corresponding del Pezzo Wahl chains for all of them. They will be type (I).

As in \cite[Proposition 5.2]{PR24}, the idea is to solve for $(\partial,n)$ the generalized Pell equation ${\partial}^2 + \ell n \partial + \ell n^2=e$ for a fixed $\ell \geq 5$ and $e\in \{-\ell,\ldots,-1\}$. We obtain recursive relations that generate all solutions $(n,\partial)$ such that $n>1$ and $\partial=(-\ell n+\sqrt{\ell^2n^2-4n^2\ell+4e})/2$. We denote them by $\left(n_{\ell, e}^{(j)}(k), \partial_{\ell, e}^{(j)}(k)\right)$, for $k \geq 0$, where the superscript $j$ distinguishes different families of solutions associated with the same pair $(\ell,e)$. We list the initial data for all the recursions
$$n_{\ell,e}^{(j)}(k)= (\ell -2) n_{\ell,e}^{(j)}(k-1)-n_{\ell,e}^{(j)}(k-2), \ \partial_{\ell,e}^{(j)}(k)=(\ell-2) \partial_{\ell,e}^{(j)}(k-1)- \partial_{\ell,e}^{(j)}(k-2)$$ where $k\geq 2$.

{\small 
\vspace{0.3cm}
\noindent
$(\ell=8)$: 

$(n_{8,-8}(0),n_{8,-7}^{(0)}(0),n_{8,-7}^{(1)}(0),n_{8,-4}(0))=(3,2,4,5)$ 

$(n_{8,-8}(1),n_{8,-7}^{(0)}(1),n_{8,-7}^{(1)}(1),n_{8,-4}(1))=(17,11,23,29)$

$(\partial_{8,-8}(0),\partial_{8,-7}^{(0)}(0),\partial_{8,-7}^{(1)}(0),\partial_{8,-4}(0))= (-4,-5,-3,-6)$

$(\partial_{8,-8}(1),\partial_{8,-7}^{(0)}(1),\partial_{8,-7}^{(1)}(1),\partial_{8,-4}(1)) = (-20,-13,-27,-34)$

\vspace{0.3cm}
\noindent
$(\ell=7)$:

$(n_{7,-5}^{(0)}(0),n_{7,-5}^{(1)}(0),n_{7,-3}(0))=(2,3,4)$ 

$(n_{7,-5}^{(0)}(1),n_{7,-5}^{(1)}(1),n_{7,-3}(1))=(9,14,19)$

$(\partial^{(0)}_{7,-5}(0),\partial^{(1)}_{7,-5}(0),\partial_{7,-3}(0))=(-3,-4,-5)$ 

$(\partial^{(0)}_{7,-5}(1),\partial^{(1)}_{7,-5}(1),\partial_{7,-3}(1))=(-11,-17,-23)$

\vspace{0.3cm}
\noindent
$(\ell=6)$: 

$(n_{6,-3}(0),n_{6,-2}(0))=(2,3)$ \ \ \ \ \ \ \ \  \ $(n_{6,-3}(1),n_{6,-2}(1))=(7,11)$

$(\partial_{6,-3}(0),\partial_{6,-2}(0)) =(-4,-4)$ \ \ \ $(\partial_{6,-3}(1),\partial_{6,-2}(1))=(-9,-14)$

\vspace{0.3cm}
\noindent
$(\ell=5)$: 

$(n_{5,-5}(0),n_{5,-1}(0))=(2,2)$ \ \ \ \ \ \ \ \ \ $(n_{5,-5}(1),n_{5,-1}(1))=(3,5)$

$(\partial_{5,-5}(0),\partial_{5,-1}(0)) =(-5,-3)$ \ \ \ $(\partial_{5,-5}(1),\partial_{5,-1}(1))=(-5,-7)$}

\vspace{0.3cm}

It follows that $\partial_{l, e}(k)=-n_{l, e}(k)-n_{l, e}(k-1)$ for $k \geq 1$. Consequently, the vector bundle $E$ is a Hacking v.b. corresponding to $\frac{1}{n_{l, e}(k)^2}\left(1, n_{l, e}(k)n_{l, e}(k-1)-1\right)$. In what follows, we describe the corresponding Wahl chains and markings of type (I) (see Remark \ref{Pellrecursion}). We recover $E$ by considering a fiber $\Gamma$ in $X$, the chain of Wahl singularities is $(0)-\left[{n_{l, e}(k) \choose n_{l, e}(k-1) } \right]$ for every $k \geq 1$. We have $\deg(E)=-n_{l, e}(k)-n_{l, e}(k-1)$. (We also have $\left[{n_{l, e}(k) \choose n_{l, e}(k-1) } \right]-(l-3)$, and for the corresponding e.c. $E',\O_Y$ we have $\deg(E')=n_{l, e}(k-1)+(1-\ell)n_{l, e}(k)$.) 

{\small 
\vspace{0.3cm}
\noindent
$(\ell=8, e=-8)$: For $k\geq 1$, %If $k=0$, then $[5,2]$. For $k \geq 1$,

\noindent
$[\underbrace{6,\dots,6}_k, 5, 3, \underbrace{2,2,2,3,\dots,2,2,2,3}_{(k-1) \ \text{blocks} \  2,2,2,3}, 2,2,2,2]$; $[\underline{6},\dots,6, 5, 1, 2,2,2,3,\dots,2,2,2,3, 2,2,2,2]$}

{\small
\vspace{0.3cm}
\noindent
$(\ell=8, e=-7, j=0)$: For $k\geq 2$,  %If $k=0$, then $[4]$. For $k \geq 1$,

\noindent
$[\underbrace{6,\dots,6}_k, 5, \underbrace{2,2,2,3,\dots,2,2,2,3}_{(k-1) \ \text{blocks} \  2,2,2,3}, 2,2,2,2]$; $[\underline{6},\dots,6, 4, 1,2,2,3,\dots,2,2,2,3, 2,2,2,2]$}

{\small
\vspace{0.3cm}
\noindent
$(\ell=8, e=-7, j=1)$: For $k\geq 1$, %If $k=0$, then $[6,2,2]$. For $k \geq 1$,  

\noindent
$[\underbrace{6,\dots,6}_k,6, 2, 3, \underbrace{2,2,2,3,\dots,2,2,2,3}_{(k-1) \ \text{blocks} \  2,2,2,3}, 2,2,2,2]$; $[\underline{6},\dots,6, 1, 2, 2,2,2,3,\dots,2,2,2,3, 2,2,2,2]$}

{\small
\vspace{0.3cm}
\noindent
$(\ell=8, e=-4)$: (Example \ref{nomarkov}) For $k\geq 1$, % If $k=0$, then $[7,2,2,2]$. For $k \geq 1$, 

\noindent
$[\underbrace{6,\dots,6}_k, 7, 2, 2, 3, \underbrace{2,2,2,3,\dots,2,2,2,3}_{(k-1) \ \text{blocks} \  2,2,2,3}, 2,2,2,2]$;

\noindent
$[\underline{6},\dots,6,7,1,2,2,2,2,2,3,\dots,2, 2,2,3,2,2,2,2]$,  $[\underline{6},\dots,6,6,2,1,3,2,2,2,3,\dots,2, 2,2,3,2,2,2,2]$}

\noindent
Each marking above produces different general fibers $\F_0$ and $\F_1$ respectively.

{\small
\vspace{0.4cm}
\noindent
$(\ell=7, e=-5, j=0)$: For $k\geq 2$,  %If $k=0$, then $[4]$. For $k \geq 1$,

\noindent
$[\underbrace{5, \dots, 5}_{k},5,\underbrace{2,2,3,\dots,2,2,3}_{(k-1) \ \text{blocks} \  2,2,3},2,2,2] $; $[\underline{5}, \dots, 5,3,1,2,3,\dots,2,2,3,2,2,2] $

\vspace{0.3cm}
\noindent
$(\ell=7, e=-5, j=1)$: For $k\geq 1$, % If $k=0$, then $[2,5]$. For $k \geq 1$,

\noindent
$[\underbrace{5, \dots, 5}_{k},5, 3, \underbrace{2,2,3,\dots,2,2,3}_{(k-1) \ \text{blocks} \  2,2,3}, 2,2,2]$; $[\underline{5},\dots,5, 4, 1, 2,2,3,\dots,2,2,3, 2,2,2]$

\vspace{0.3cm}
\noindent
$(\ell=7, e=-3)$: For $k\geq 1$, % If $k=0$, then $[6,2,2]$. For $k \geq 1$,

\noindent
$[\underbrace{5, \dots, 5}_k, 6, 2, 3, \underbrace{2,2,3,\dots,2,2,3}_{(k-1) \ \text{blocks} \  2,2,3}, 2,2,2] $; $[\underline{5}, \dots, 5, 5, 1, 2, 2,2,3,\dots,2,2,3, 2,2,2]$

\vspace{0.3cm}
\noindent
$(\ell=6, e=-3)$: For $k\geq 2$, % If $k=0$, then $[4]$. For $k \geq 1$,

\noindent
$[\underbrace{4, \dots, 4}_k, 5, \underbrace{2,3,\dots,2,3}_{(k-1) \ \text{blocks} \  2,3}, 2,2] $; $[\underline{4}, \dots, 4, 2, 1,3,\dots,2,3, 2,2]$

\vspace{0.3cm}
\noindent
$(\ell=6, e=-2)$: For $k\geq 1$, % If $k=0$, then $[6,2,2]$. For $k \geq 1$,

\noindent
$[\underbrace{4, \dots, 4}_k, 5, 3, \underbrace{2,3,\dots,2,3}_{(k-1) \ \text{blocks} \  2,3}, 2,2]$; $[\underline{4}, \dots, 4, 3, 1, 2,3,\dots,2,3, 2,2]$

\vspace{0.3cm}
\noindent
$(\ell=5, e=-5)$: For $k\geq 2$,  %If $k=0$, then $[4]$. For $k \geq 1$,

\noindent
$[\underbrace{3, \dots, 3}_{k-1}, 2, 6, \underbrace{3,\dots,3}_{k-2}, 2] $; $[\underline{3}, \dots, 3, 2, 1, 3,\dots,3, 2] $}

\vspace{0.3cm}
\noindent
$(\ell=5, e=-1)$: For $k\geq 2$, %If $k=0$, then $[4]$. For $k \geq 1$,

\noindent
$[\underbrace{3, \dots, 3}_k, 5, \underbrace{3,\dots,3}_{k-1}, 2] $; $[\underline{3}, \dots, 3, 1, 2, 3,\dots,3, 2] $

\label{AllPolishchuk}
\end{remark}

%--- In relation to that problem, let us conjecture that, in such a case, there must be a type (I). What do we have then? Say something explicit. If we have markings in tails, as in canonical degree $4$, then maybe we have density as before. CHECK THIS OUT!

%\bigskip 
%--- \textbf{Problem}: What is the effect of a Mori/Markov mutation? I think it is a mutation of vector bundles. The idea is that the exceptional collection is preserved, the new exceptional object must be a shifted H.v.b., and the rank and degree coincide. CAREFUL HERE, one has to go step by step, moving H.v.b. etc. See \cite{TU22} Theorem 6.4, Lemma 6.3, etc. 

%--- We have $2$ mutations: Mori generalized (including change of sign i.e. flip and antiflip and Markov mutation.

%--- Start with a Toric surface $W$ with $K_W^2 >0$ and only Wahl singularities, and $W_t \rightsquigarrow W$ with $W_t$ del Pezzo. 

%--- We have a H.e.c. $E_r,\ldots,E_0$ associated to a chain of Wahl singularities with $\Gamma_1,\ldots, \Gamma_r$. 

%--- Take $\Gamma_i$. 

%--- If $\Gamma_i \cdot K >0$, then Mori antiflip.

%--- If $\Gamma_i \cdot K =0$, then flop (nothing really on the surface).

%--- If $\Gamma_i \cdot K <0$ and $\Gamma^2 >0$, then Markov mutation.

%--- If $\Gamma_i \cdot K <0$ and $\Gamma^2 <0$, then Mori mutation or flip.

%--- For each of these cases we prove that the corresponding mutation of vector bundles holds. 

%--------------------------------------------------------------------------------
\bibliographystyle{plain}
\bibliography{math}

\end{document}